\documentclass{article}

\usepackage
[
colorlinks=true,
linkcolor=blue,
anchorcolor=blue,
citecolor=blue,
urlcolor=blue,
plainpages=false,
pdfpagelabels
]{hyperref}

\usepackage{graphicx}
\usepackage{enumerate}
\usepackage{amsmath}
\usepackage{bbm}
\usepackage{mathrsfs}
\usepackage{amsfonts}
\usepackage[marginpar]{todo}
\usepackage{amsthm}
\usepackage[numbers]{natbib}
\usepackage{color}
\newcounter{thm}
\newtheorem{theorem}[thm]{Theorem}
\newtheorem{definition}[thm]{Definition}
\newtheorem{corollary}[thm]{Corollary}
\newtheorem{lemma}[thm]{Lemma}
\newtheorem*{remark}{Remark}
\newtheorem*{example}{Example}

\newcommand{\tup}{\textup}

\newcommand\R{{\mathbb{R}}}
\newcommand\Z{{\mathbb{Z}}}

\newcommand\Stwo{{\mathbb{S}^2}}
\newcommand\Sone{{\mathbb{S}^1}}

\newcommand{\G}{\mathbb{G}}
\newcommand{\C}{\mathbb{C}}
\newcommand{\Quat}{\mathbb{H}}

\newcommand{\qOne}{\ensuremath{\mathbbm{1}}}
\newcommand{\qi}{\ensuremath{\mathbbm{i}}}
\newcommand{\qj}{\ensuremath{\mathbbm{j}}}
\newcommand{\qk}{\ensuremath{\mathbbm{k}}}
\newcommand{\qf}{\ensuremath{\mathbbm{f}}}
\newcommand{\qg}{\ensuremath{\mathbbm{g}}}

\newcommand{\spann}{\mathop{\mathrm{span}}}

\newcommand{\foh}{\frac{1}{2}}

\newcommand{\I}{\mathrm{I}}
\newcommand{\II}{\mathrm{II}}
\newcommand{\III}{\mathrm{III}}
\newcommand{\K}{\mathcal{K}}
\newcommand{\Hc}{\mathcal{H}}
\newcommand{\So}{\mathcal{S}}

\newcommand{\mattwo}[4]
{\left(\begin{array}{cc}
                        #1  & #2   \\
                        #3 &  #4
                   \end{array}\right) }

\newcommand{\tr}{\mathop{\mathrm{tr}}}

\definecolor{color1}{rgb}{0.7725490196078432, 0.8745098039215686, 1.0} 
\definecolor{color2}{rgb}{0.11764705882352941, 0.3254901960784314, 1.0} 
\definecolor{color3}{rgb}{0.023529411764705882, 0.0392156862745098, 0.7411764705882353} 
\definecolor{color4}{rgb}{1.0, 0.8980392156862745, 0.7725490196078432} 
\definecolor{color5}{rgb}{1.0, 0.4196078431372549, 0.12156862745098039} 
\definecolor{color6}{rgb}{0.7411764705882353, 0.11764705882352941, 0.023529411764705882} 
\definecolor{color7}{rgb}{1.0, 0.8980392156862745, 0.7725490196078432} 
\definecolor{color8}{rgb}{1.0, 0.792156862745098, 0.11764705882352941} 
\definecolor{color9}{rgb}{0.7411764705882353, 0.7215686274509804, 0.023529411764705882} 
\definecolor{color10}{rgb}{0.7725490196078432, 0.8745098039215686, 1.0} 
\definecolor{color11}{rgb}{0.12156862745098039, 0.6980392156862745, 1.0} 
\definecolor{color12}{rgb}{0.023529411764705882, 0.6470588235294118, 0.7411764705882353} 
\hypersetup{
  colorlinks,
  citecolor=color2,
  linkcolor=color3,
  urlcolor=color6}

\title{A discrete parametrized surface theory in $\R^3$}
\author{Tim Hoffmann\thanks{This research was supported by the DFG-Collaborative Research Center, TRR 109, ``Discretization in Geometry and Dynamics.''}, Andrew O. Sageman-Furnas, Max Wardetzky\thanks{This research was partially supported by the DFG-Collaborative Research Center, TRR 109, ``Discretization in Geometry and Dynamics.''}}
\begin{document}
\maketitle
\begin{abstract}
We propose a discrete surface theory in $\R^3$ that unites the most prevalent versions of discrete special parametrizations. This theory encapsulates a large class of discrete surfaces given by a Lax representation and, in particular, the one-parameter associated families of constant curvature surfaces. The theory is not restricted to integrable geometries, but extends to a general surface theory.
\end{abstract}AMS 2010 subject classification: primary 53A05; secondary 52C99\\

\section{Introduction}
A \emph{quad net} is a map from a strongly regular polytopal cell decomposition of a regular surface with all faces being quadrilaterals into $\R^3$ with nonvanishing straight edges. Notice, in particular, that nonplanar faces are admissible. In \emph{discrete differential geometry} quad nets are understood as discretizations of parametrized surfaces \cite{Bobenko:1999wj,Bobenko:2008tn,Bobenko:2008uh}. In this agenda many classes of special surfaces have been discretized using algebro-geometric approaches for integrable geometry---originally using discrete analogues of soliton theory techniques (e.g., discrete Lax pairs and finite-gap integration \cite{Bobenko:1996ug}) to construct nets, but more recently using the notion of 3D consistency (reviewed in \cite{Bobenko:2008tn}).\footnote{As in the smooth setting, these approaches have been successfully applied to space forms (see, e.g., \cite{Hoffmann:1999vb,Bobenko:2014cp,Burstall:2014dr,Burstall:2014tg}).} As an example consider the case of K-surfaces (i.e., surfaces of constant negative Gau{\ss} curvature). The integrability equations of classical surface theory are equivalent to the famous sine-Gordon equation \cite{Bianchi:1902vd,Darboux:1887tq}. In an integrable discretization the sine-Gordon equation becomes a finite difference equation for which integrability is encoded by a certain closing condition around a 3D cube. Both in the smooth and discrete setting, integrability is bound to specific choices of parameterizations, such as asymptotic line parametrizations for K-surfaces. In this way different classes of surfaces, such as minimal surfaces or surfaces of constant mean curvature, lead to different PDEs and give rise to different parameterizations. In the discrete case, this is reflected by developments that treat different special surfaces by disparate approaches. These integrable discretizations maintain characteristic properties of their smooth counterparts (e.g., the transformation theory of Darboux, B\"acklund, Bianchi, etc.) and sometimes give rise to satisfying self-contained theories \cite{Bobenko:2006ts} within the special classes that they consider. What has been lacking, however, is a unified discrete theory that lifts the restriction to special surface parametrizations. Indeed, different from the case of classical smooth surface theory, existing literature does not provide a general discrete theory for quad nets.

We propose a theory that encompasses the most prevalent versions of existing discrete special parametrizations (reviewed in \cite{Bobenko:1999us}), such as discrete conjugate nets \cite{Sauer:1970tk}, discrete (circular) curvature line nets \cite{Nutbourne:1988tg,Cieslinski:1997ub,Bobenko:2008hm}, discrete isothermic nets \cite{Bobenko:1996vq,Bobenko:2009tt}, and discrete asymptotic line nets \cite{Sauer:1950ca,Wunderlich:1951wc}. Our approach provides a curvature theory that, in particular, yields appropriate curvatures for previously defined discrete minimal \cite{Bobenko:1996vq}, discrete constant mean curvature (cmc) \cite{Pedit:1995wp,Bobenko:1999us,Hoffmann:1999vm}, discrete constant negative Gau{\ss} curvature \cite{Sauer:1950ca,Wunderlich:1951wc,Bobenko:1996ug,Hoffmann:1999wq,pinkall2008designing}, and discrete developable surfaces \cite{Liu:2006tm}. This theory not only retrieves the curvature definitions given in \cite{Schief20061484,Bobenko:2010eg} in the case of planar faces but extends to the general setting of nonplanar quads. Moreover, for the first time, it provides a way to understand the one-parameter \emph{associated families} of discrete surfaces of constant curvature, both in terms of discrete curvature and discrete conformality.

The fundamental property of our approach is the following \emph{edge-constraint} that couples discrete surface points and normals: the average normal along an edge is perpendicular to that edge. This condition arises from a Steiner-type (i.e., offset and mixed area) perspective on curvature and, while surprisingly elementary, has profound consequences for the theory. By introducing a Gau{\ss} map for general nonplanar quad nets, our theory builds on basic construction principles of the classical smooth setting.

The paper is organized as follows: after the definition of edge-constraint nets (Section \ref{sec:edgeConstraint}) we introduce their curvatures, naturally extending the work
of Schief \cite{Schief20061484} and Bobenko, Pottmann, and Wallner \cite{Bobenko:2010eg}. These curvatures are then shown to be consistent with first, second, and third fundamental forms for edge-constraint nets. We describe the classical discrete integrable surfaces of constant curvature (circular minimal, circular cmc, and asymptotic and circular K-nets) and show that they are indeed edge-constraint nets of constant curvature (Section \ref{sec:constantCurvatureNets}). Even more, we show that they possess associated families that are also edge-constraint nets exhibiting constant curvature. In particular, the proof for cmc nets shows a rather unexpected connection between their 3D compatibility cube and the general Bianchi permutability cube for discrete curves \cite{Hoffmann:2008ub,TabachnikovBicycleMap}. The section closes with a discussion of discrete developable nets. We then provide a short treatment on how discrete conformality is represented in our theory (Section \ref{sec:conformal}), showing that the members of the associated family of minimal nets are conformally equivalent. We conclude (Section \ref{sec:LaxPairedgeConstraint}) by showing that a rather general class of nets generated by a Sym--Bobenko formula is in fact a subset of edge-constraint nets.

\section{Edge-constraint nets}
\label{sec:edgeConstraint}
\subsection{Setup}
A natural discrete analogue of a parametrized surface patch is a map from $\Z^2 \to \R^3$ corresponding to a single chart. To consider discrete atlases we relax the combinatorial restrictions and think more generally of maps from \emph{quadrilateral graphs}. We will use the words quadrilateral and quad interchangeably.

\begin{definition}[Quadrilateral net]
A \emph{quad graph} $\G$ is a strongly regular polytopal cell decomposition of a regular surface with all faces being quadrilaterals. A map $f: \G \to \R^3$ is called a \emph{(quad) net}.
\end{definition}

\begin{figure}[ht]
  \centering
  \includegraphics[width=0.7\hsize]{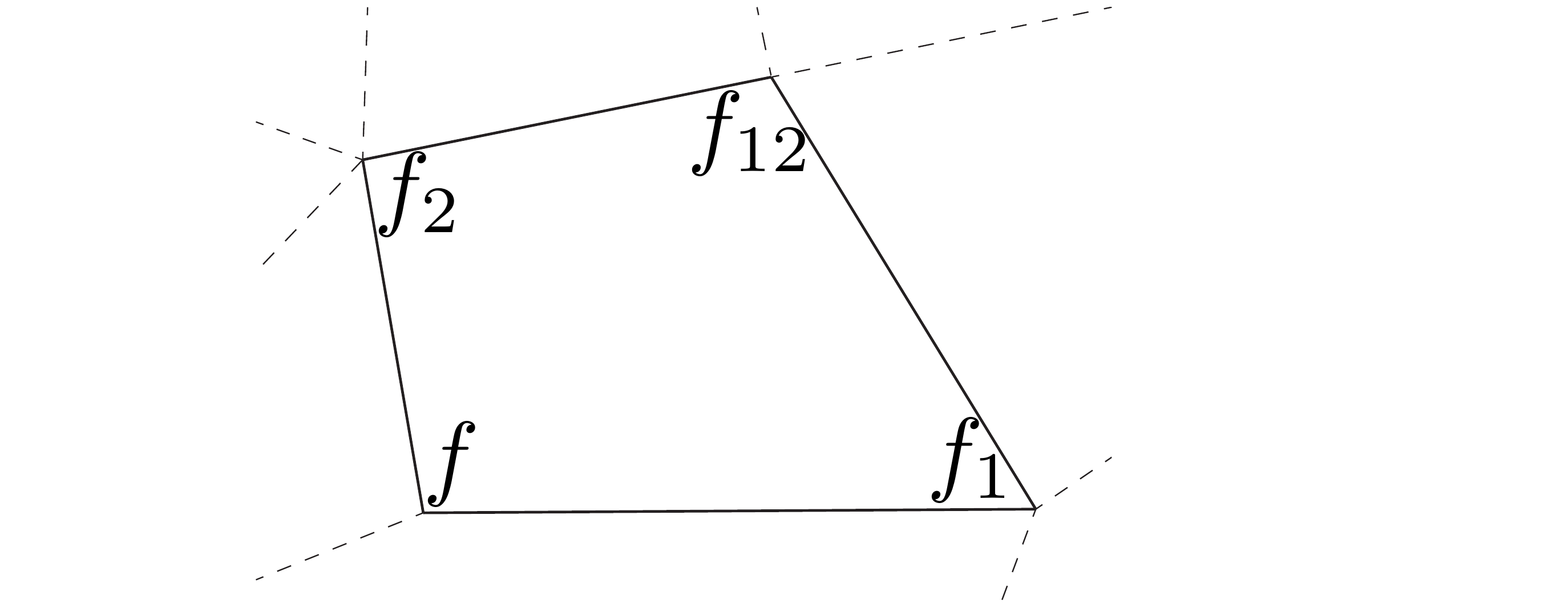}
  \caption{Shift notation used to describe the points of a quad net $f$, even with irregular combinatorics.}
  \label{fig:shiftNotation}
\end{figure}

\begin{remark}[Shift notation]
As seen in Figure \ref{fig:shiftNotation}, we will use \emph{shift notation} \cite{Bobenko:2008tn} to describe the points of a quad net: when the underlying quad graph has the combinatorics of $\Z^2$, we denote a point by $f = f_{k,l}$ for some $k, l \in \Z$ and define the shift operators $f_1 := f_{k+1, l}$ and $f_2 := f_{k, l+1}$. The point diagonal to $f$ is given by a shift in each direction, $f_{12} := f_{k+1, l+1}$. In what follows we do not restrict ourselves to the combinatorics of $\Z^2$, but will continue to use shift notation, as there is no ambiguity when the discussion is restricted to a point $f$; oriented edge $f_i - f$ with $i = 1,2$; or quad $(f, f_1, f_{12}, f_2)$.
\end{remark}

Immersed parametrized surfaces in the smooth setting can be thought of either as a smooth family of points or as the envelope of a family of tangent planes. One defines a \emph{contact element} at a point $p \in \R^2$ of a parametrized surface $f: D \subset \R^2 \to \R^3$ as the pair $(f(p), P(p))$, consisting of a point $f(p)$ and the oriented tangent plane $P(p)$ passing through it. $P(p)$ is completely determined by its unit normal $n(p)$ when anchored at $f(p)$; considering $n(p)$ at the origin defines the \emph{Gauss map} $n: D \subset \R^2 \to \Stwo$. Using this perspective we consider parametrized surfaces as the pair of maps $(f, n)$, an immersion together with its Gau{\ss} map. 

Analogously we consider discrete parametrized surfaces not as a single quad net, but as a pair of nets that are weakly coupled, mimicking the relationship between an immersion and its Gau{\ss} map in the smooth setting. This pair will be known as an \emph{edge-constraint net} and is our main object of study.

\begin{definition}[Edge-constraint net]
Let $\G$ be a quad graph. We call a pair of quadrilateral nets $(f,n): \G \to \R^3 \times \Stwo$ a \emph{contact element net}. A contact element net is called an \emph{edge-constraint net} if it satisfies the following:
\begin{description}
\item[Edge-constraint] For each pair of points of $f$ connected by an edge, the average of the normals at those points is perpendicular to the edge, i.e., for $i=1,2$ we have $f_i - f \perp \foh(n_i + n)$.
\end{description}
We further assume that $f$ contains no vanishing edges, i.e., that $f_i - f$ is always nonzero.

The maps $f:\G \to \R^3$ and $n: \G \to \Stwo$ are called the \emph{(discrete) immersion} and \emph{Gau{\ss} map}, respectively.
\end{definition}

\begin{remark}
As edge-constraint nets are in fact a pair of nets, unless we state explicitly that we are referring only to the immersion $f$ or the Gau{\ss} map $n$, the combinatorial language of \emph{vertex, edge, face} (or \emph{quad}) will refer to the combinatorics of the underlying quad graph $\G$.
\end{remark}

The edge-constraint discretizes a coupling between the Gau{\ss} map and immersion; in the smooth setting it is generic in the following sense.

\begin{lemma}
Let $f:\R^2 \to \R^3$ parametrize a smooth surface patch with Gau{\ss} map $n: \R^2 \to \Stwo$. For every point $p \in \R^2$ and unit vector $v \in \R^2$, let the images of the line $p + t v$ where $t \in \R$ be given (with a slight abuse of notation) by $f(t) := f(p + t v)$ and $n(t) := n(p + tv)$, respectively. Then the central and one-sided difference approximations to the edge-constraint along $f(t)$ are satisfied up to second order, i.e., 
\begin{eqnarray}
(f(\epsilon) - f(- \epsilon)) \cdot (n(\epsilon) + n(- \epsilon)) &=& 0 +  O(\epsilon^3) ~ \mathrm{and}~  \nonumber \\
(f(\epsilon) - f(0)) \cdot (n(\epsilon) + n(0)) &=& 0 +  O(\epsilon^3). \nonumber
\end{eqnarray}
\begin{proof}
Note that $f'(t) \cdot n(t) = 0$ by construction, so in particular $f''(t) \cdot n(t) + f'(x) \cdot n'(t) = 0$. The statement then follows by Taylor expanding $f(t)$ and $n(t)$ around $t=0$ with small parameter $\epsilon > 0$.
\end{proof}
\end{lemma}

The simplest class of edge-constraint nets are those given by quadrilateral nets in spheres. 

\begin{lemma}
\label{lem:sphericalNets}
Let $f: \G \to r \Stwo$ be a quad net in the sphere of radius $r > 0$. Then $f$ together with the Gau{\ss} map $n = f/r$ is an edge-constraint net.
\end{lemma}

Edge-constraint nets naturally exhibit offset nets by adding multiples of the Gau{\ss} map to the original immersion while keeping the Gau{\ss} map fixed. This observation is the foundation of their curvature theory. 

\begin{lemma}[Offset nets]
\label{lem:offsetNets}
For any $t \in \R$ and contact element net $(f,n):\G \to \R^3 \times \Stwo$, the contact element net $(f + t\,n, n)$, where linear combinations are taken on vertices, is an edge-constraint net if and only if $(f, n)$ is an edge-constraint net.
\end{lemma}

\subsection{Curvatures from Offsets}
\label{sec:curvatureTheory}
Let $f: \R^2 \to \R^3$ be a smooth parametrized surface with Gau{\ss} map $n: \R^2 \to \R^3$. For each $t \in \R$ and $(x,y) \in \R^2$ we define the offset surface by $f^t(x,y) := f(x,y) + t \, n(x,y)$. We only consider smooth parametrizations $f$ which give rise to smooth offsets $f^t$ for small enough $|t|$. It is easily seen that $n$ is also the Gau{\ss} map for the offset surface $f^t$, thus the construction of offset nets seen in Lemma \ref{lem:offsetNets}. For all $(x,y) \in \R^2$, the area element at $f^t(x,y)$ can be expressed in terms of the area element, mean and Gau{\ss} curvatures at $f(x,y)$. This relationship is known as the \emph{Steiner formula} and is best understood through the \emph{mixed area form}.

\begin{definition}[Mixed area form]
Let $g,h: \R^2 \to \R^3$ parametrize two smooth surfaces that share a Gau{\ss} map $N: \R^2 \to \Stwo$. For every $(x,y) \in \R^2$, we define the \emph{mixed area form} in the tangent plane $P \perp N(x,y)$ by
\begin{equation}
\label{eq:curvatureForms}
A(g,h) := \foh(\det(g_x, h_y, N) + \det(h_x, g_y, N)),
\end{equation}
where subscripts denote partial derivatives. When $g = h$, the mixed area form reduces to the area element of $g$.
\end{definition}

\begin{remark}
To define the mixed area form we switched notation to a capital $N,$ as opposed to little $n$, for the Gau{\ss} map. While in the smooth setting these two objects coincide, the discrete Gau{\ss} map for an edge-constraint net (also denoted by $n$) lives on vertices, whereas we will define the mixed area form for an edge-constraint net on faces. We will define a new unit vector per face, which we call the \emph{projection direction} (denoted by $N$), that defines the tangent plane $P$ where this mixed area form lives. 
\end{remark}

We can now state the Steiner formula and consequently define the mean and Gau{\ss} curvature functions on a smooth parametrized surface. The same definitions will carry over to edge-constraint nets.

\begin{theorem}[Steiner formula]
Let $f: \R^2 \to \R^3$ be a smooth surface parametrization with Gau{\ss} map $n: \R^2 \to \R^3$ and offset surface $f^t: \R^2 \to \R^3$. Then for each $(x,y) \in \R^2$ the following relationship holds 
\begin{eqnarray}
\label{eq:steinerFormula}
A(f^t, f^t) &=& A(f, f) + 2 t A(f,n) + t^2 A(n,n), \nonumber \\
&=& (1 + 2 t \Hc + t^2 \K) A(f,f), 
\end{eqnarray}
defining
\begin{equation}
\label{eq:meanGaussCurvature}
\Hc := \frac{A(f,n)}{A(f,f)} ~\mathrm{and}~ \K := \frac{A(n,n)}{A(f,f)},
\end{equation}
 as the mean and Gau{\ss} curvature functions, respectively.
\end{theorem}

\begin{figure}[t]
  \centering
  \includegraphics[width=.50\hsize]{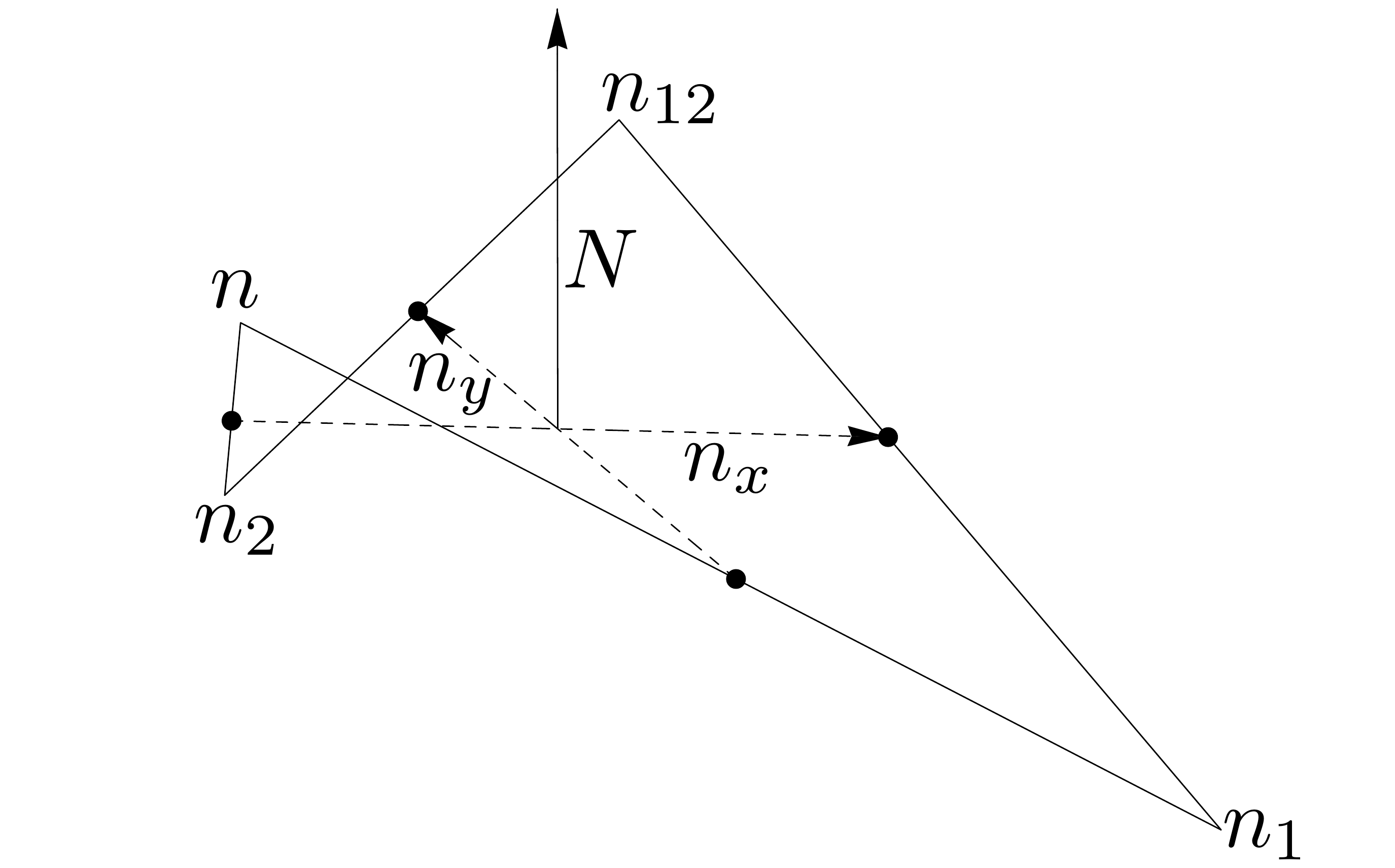}
  \caption{The partial derivatives, $n_x$ and $n_y$, of a (possibly non-planar) quadrilateral from the Gau{\ss} map $n$ of an edge-constraint net as defined by the midpoint connectors. A \emph{projection direction}, $N \perp n_x, n_y$ is also shown, determining the face tangent plane $P$ where the mixed area form is defined.}
  \label{fig:partialNotation}
\end{figure}

\begin{definition}
\label{def:partialsBigNCurvatures}
Consider a single quadrilateral from an edge-constraint net $(f,n)$. We define the partial derivatives as the midpoint connectors of the (possibly non-planar) quadrilaterals for each of $f$ and $n$, e.g., for the Gau{\ss} map, as shown in Figure \ref{fig:partialNotation}, we have
\begin{equation}
\label{eq:partialDerivatives}
n_x := \foh(n_{12} + n_1) - \foh(n_{2} + n), ~ n_y := \foh(n_{12} + n_2) - \foh(n_1 + n),
\end{equation}
and likewise for $f$.
Then the set of admissible \emph{projection directions} is defined as
\begin{equation}
U := \{N \in \Stwo | ~ N \perp \spann{\{n_x,n_y\}}\}.
\end{equation}
Generically, the projection direction is unique (up to sign) and we choose
\begin{equation}
N := \frac{n_x \times n_y}{\| n_x \times n_y \|}.
\end{equation}
In the \emph{(quad) tangent plane} $P \perp N$ we define the \emph{(discrete) mixed area form} via Equation (\ref{eq:curvatureForms}), yielding a Steiner formula (Equation (\ref{eq:steinerFormula})).
\end{definition}

\begin{remark}[Degenerate Gau{\ss} maps]
The set of admissible projection directions also generates a consistent curvature theory in the degenerate situation where the partial derivatives of the Gau{\ss} map are not linearly independent. Degenerate Gau{\ss} maps naturally arise in the theory of developable surfaces, so we defer this discussion to the theory of developable edge-constraint nets in Section \ref{sec:developable}.
\end{remark}
\begin{definition}[Edge-constraint curvatures] 
Using the Steiner formula (Equation (\ref{eq:steinerFormula})) we define \emph{mean} and \emph{Gau{\ss} curvatures} for edge-constraint nets via Equation (\ref{eq:meanGaussCurvature}).
\end{definition}

\begin{remark}[Degenerate immersions]
Clearly the curvatures are only well defined when the immersion has non-vanishing area $A(f,f) \neq 0$. Every statement we make about curvatures will assume the immersion quad has non-vanishing area.
\end{remark}
\begin{remark}[Sign of projection direction]
The sign of the projection direction in the generic setting does not correspond to a change in local orientation. The mixed area form will obviously change sign, but the mean and Gau{\ss} curvatures are invariant to this choice. However, flipping the Gau{\ss} map ($n \to -n$) does change the sign of the mean curvature, as expected.
\end{remark}
\begin{remark}[Choice of partial derivatives]
The choice to define the partial derivatives as the midpoint connectors is to guide intuition. In fact, one has the freedom to choose \emph{any linear combination} of the midpoint connectors to be the partial derivatives, as long as \emph{the same combination} is chosen for both the immersion and the Gau{\ss} map. This corresponds to the freedom to locally reparametrize in the smooth setting. The mean and Gau{\ss} curvatures and forthcoming definitions of principal curvatures and curvature line fields are all invariant to this choice. As expected, the mixed area forms and fundamental forms will change as these are not invariant to a local reparametrization in the smooth setting either.
\end{remark}
In the smooth setting one can also derive the mean and Gau{\ss} curvatures at a point via the fundamental forms and shape operator living in the tangent plane to that point. The shape operator is a real self-adjoint bilinear form whose eigenvalues and eigenvectors are the principal curvatures and curvature lines, respectively. From Definition \ref{def:partialsBigNCurvatures} we can discretize the fundamental forms and shape operator in the plane perpendicular to the projection direction of each face of an edge-constraint net.

\begin{definition}[Fundamental forms]
Consider a single quad from an edge-constraint net $(f,n)$. Let $\pi$ be the projection into the quad tangent plane P. Set $\hat f_x := \pi(f_x) = f_x - (f_x \cdot N) N$ and similarly for $\hat f_y$.\footnote{Since $N \perp n_x, n_y$ by Definition \ref{def:partialsBigNCurvatures}, $\hat n_x = n_x$ and $\hat n_y = n_y$.} We define the \emph{fundamental forms} and \emph{shape operator} by:
\begin{eqnarray}
\I &:= \mattwo{\hat f_x \cdot \hat f_x}{\hat f_x \cdot \hat f_y}{\hat f_y \cdot \hat f_x}{\hat f_y \cdot \hat f_y}, ~~ \II  &:= \mattwo{\hat f_x \cdot n_x}{\hat f_x \cdot n_y}{\hat f_y \cdot n_x}{\hat f_y \cdot n_y}, \nonumber \\
\III &:= \mattwo{n_x \cdot n_x}{n_x \cdot n_y}{n_y \cdot n_x}{n_y \cdot n_y},~~\So &:= ~  \I^{-1} \II.
\end{eqnarray}
The eigenvalues $(k_1, k_2)$ and eigenvectors of the shape operator $S$ are the \emph{principal curvatures} and \emph{curvature line fields} of the quadrilateral.
\end{definition}

The existence of principal curvatures and curvature line fields follows from the symmetry of the second fundamental form:

\begin{lemma}
Consider a single quad from an edge-constraint net $(f,n)$, then the second fundamental form is symmetric.
\begin{proof}
In the above notation we want to show $\hat f_x \cdot n_y - n_x \cdot \hat f_y = 0$. As $N \perp n_x, n_y$ this quantity is the same when unprojected, i.e., 
\begin{equation}
\hat f_x \cdot n_y - n_x \cdot \hat f_y = f_x \cdot n_y - n_x \cdot f_y.
\end{equation}
Expanding out $f_x \cdot n_y - n_x \cdot f_y$ one finds it is a constant multiple of the sum of the edge-constraint conditions once around the quadrilateral, which vanishes as it vanishes on each edge by assumption.
\end{proof}
\end{lemma}

The mean and Gau{\ss} curvatures per quadrilateral defined via the Steiner formula are equal to the ones derived from the eigenvalues of the shape operator.

\begin{lemma}[Curvature and fundamental form relationships]
\label{lem:curvaturesAndFundamentalForms}
The following relations hold true in the smooth and discrete case:\\
\indent 1. $\K = k_1 \, k_2 = \det\II / \det\I$,\\
\indent 2. $\Hc = \foh \left( k_1 + k_2\right)$,\\
\indent 3. $\III - 2 \Hc\, \II + \K\, \I = 0$, and\\
\indent 4. $A(f,f)^2 = \det\I$.
\end{lemma}

\begin{example}[Spherical edge-constraint nets]
Let $(f, f/r)$ be an edge-constraint net in the sphere of radius $r > 0$ as determined by Lemma \ref{lem:sphericalNets}. Then every quadrilateral has the expected Gau{\ss} ($\K = \frac{1}{r^2}$) and mean curvature ($\Hc = \frac{1}{r}$).
\end{example}

\begin{example} [Curvature line fields]
Figure \ref{fig:ellipsoidCurvatureLines} shows the curvature line fields of an ellipsoid in the smooth and discrete setting.
\begin{figure}[ht]
  \centering
  \includegraphics[width=\hsize]{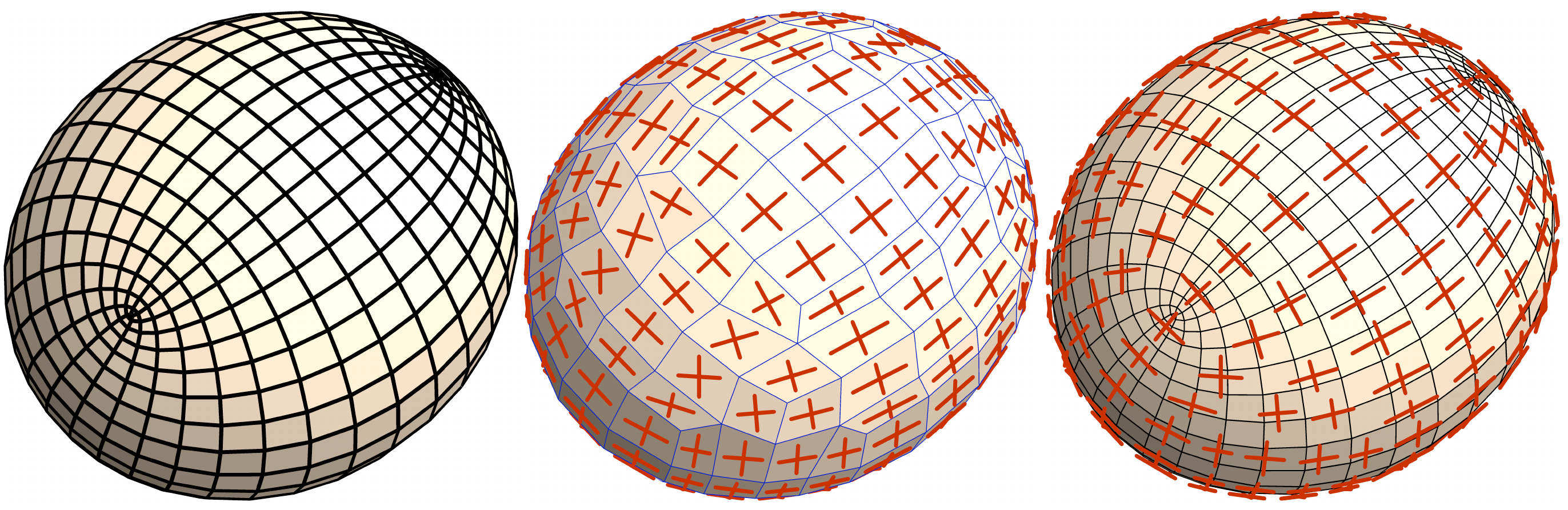}
  \caption{Left: Curvature lines of a smooth ellipsoid. Middle: Curvature line field of an ellipsoid edge-constraint net. Right: Overlay of smooth and discrete curvature line fields.}
  \label{fig:ellipsoidCurvatureLines}
\end{figure}
\end{example}

\section{Constant Curvature Nets}
\label{sec:constantCurvatureNets}
Edge-constraint nets and their curvature theory provide a unifying geometric framework through which to understand previously defined notions of discrete surfaces of constant curvature in special parametrizations. Due to their governing integrable structure, these surfaces naturally arise in one parameter \emph{associated families} that in the smooth setting fix the respective curvature, but change the type of parametrization. Previous notions of discrete curvature exist for each particular type of special parametrization, but have been difficult to reconcile with the corresponding (differently parametrized) associated families.

In what follows we rectify these discrepancies by showing that the algebraically constructed discrete isothermic minimal surfaces \cite{Bobenko:1996vq}; discrete isothermic constant mean curvature surfaces \cite{Bobenko:1996vq}; discrete asymptotic line constant negative Gau{\ss} curvature surfaces \cite{Bobenko:1996ug}; and discrete curvature line constant negative Gau{\ss} curvature surfaces \cite{Konopelchenko:1999te}, \emph{together with each of their respective associated families} are in fact edge-constraint nets with their respective curvatures constant. To close we introduce a theory of developable edge-constraint nets, a non-integrable example.

\subsection{Discrete minimal surfaces}
\label{sec:minimal}
\begin{figure}
  \centering
  \includegraphics[width=.45\hsize]{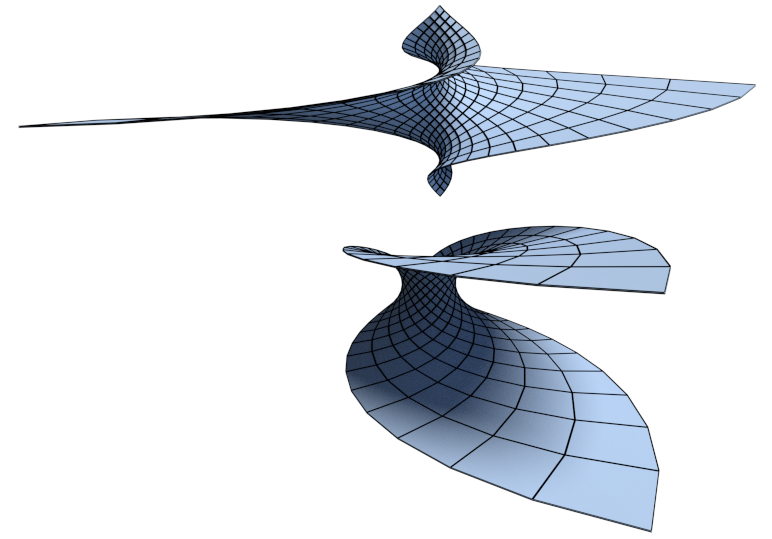}
  \includegraphics[width=.5\hsize]{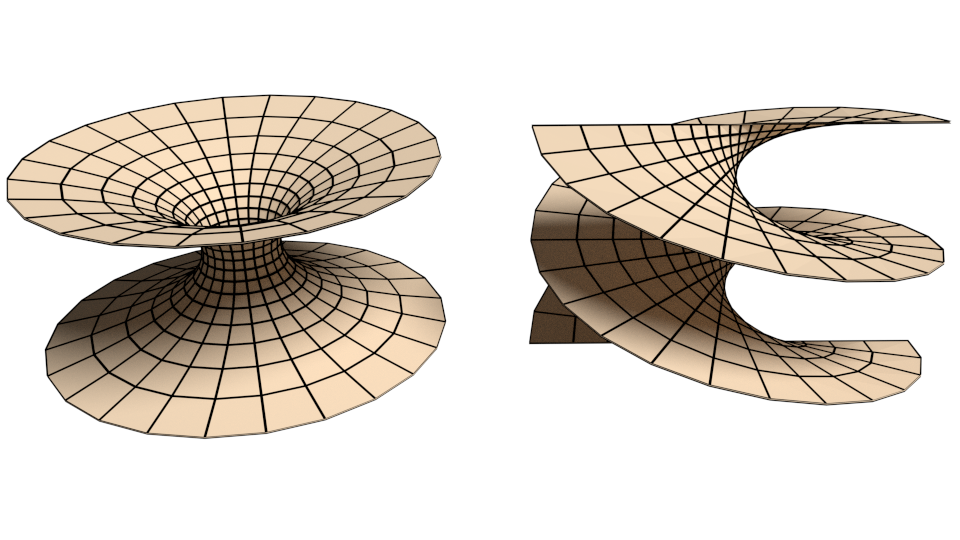}
  \caption{Two pairs of minimal edge-constraint nets from the same associated family generated from Weierstrass data. Left: Circular net helicoid and an A-net catenoid. Right: Circular net catenoid and A-net helicoid.}
  \label{fig:minAssocFamily}
\end{figure}

We start with the general definition.

\begin{definition}[Minimal edge-constraint net]
An edge-constraint net $(f,n)$ is called \emph{minimal} if every quad has vanishing mean curvature ($\Hc = 0$), i.e., the mixed area $A(f,n)$ vanishes.
\end{definition}

In the smooth setting, minimal surfaces are often parametrized by isothermic (curvature line and conformal) coordinates arising naturally form their construction from holomorphic Weierstrass data: Stereographically project a holomorphic function $g: \C \to \C$ onto the Riemann sphere to get a conformal map $n: \C \to \Stwo$. Now, think of $n$ as the Gau{\ss} map to a surface and construct the \emph{Christoffel dual} isothermic surface $f: \C \to \R^3$ by integrating
\begin{equation}
\label{eq:smoothDual}
f_x(x,y) = \frac{n_x(x,y)}{\|n_x(x,y)\|^2} \mathrm{~~and~~} f_y(x,y) = -\frac{n_y(x,y)}{\|n_y(x,y)\|^2}.
\end{equation}
The resulting $f$ is an isothermic parametrization of a minimal surface in $\R^3$ with Gau{\ss} map given by the conformal map $n$. This process of generating a minimal surface is called the Weierstrass representation. 

Bobenko and Pinkall defined discrete minimal surfaces as a special case of discrete isothermic surfaces and showed they exhibit a discrete Weierstrass representation \cite{Bobenko:1996vq}. These nets indeed have vanishing mean curvature in a curvature theory for nets with planar faces (that in the case of contact element nets is contained in the present theory) \cite{Schief20061484, Bobenko:2010eg}.

In complete analogy to the smooth case, one can extend this representation into an associated family.
This corresponds to locally rotating the frame, therefore changing the type of parametrization away from being curvature line (while staying conformal in the smooth setting). While this is an algebraic way to define the discrete nets of the associated family there has been no notion through which one can understand their minimality. The goal of this section is to rectify this by showing that every member of the associated family is an edge-constraint net and that its mean curvature vanishes on every quad.

Formulating the discrete Weierstrass representation requires discrete analogues of curvature line pa\-ra\-me\-tri\-zations, Christoffel duals, and isothermic pa\-ra\-me\-tri\-zations. We briefly introduce these notions, but emphasize that each of these discrete objects is interesting in its own right (see the book by Bobenko and Suris \cite{Bobenko:2008tn}).

\begin{definition} [Circular net]
A contact element net $(f,n)$ is called a \emph{circular net} or \emph{discrete curvature line net} if:\\
\indent 1. every quad of the immersion $f$ is \emph{circular}, its vertices lie on a circle; and\\
\indent 2. the Gau{\ss} map along each edge is found by reflection through the immersion edge perpendicular bisector plane, i.e., for $i = 1,2$ we have 
\begin{equation}
\label{eq:circSym}
n = n_i - 2 \frac{n_i \cdot (f_i - f)}{\|f_i - f\|^2} (f_i - f).
\end{equation}
\end{definition}
\begin{lemma}
Let $(f,n)$ be a circular net, then it is an edge-constraint net.
\end{lemma}
\begin{proof}
Equation (\ref{eq:circSym}) gives: $f_i-f \parallel n_i-n \perp n_i+n$.
\end{proof}
\noindent Note that the symmetry imposed by the second property implies that the Gau{\ss} map and all offset nets $(f + tn, n)$ are also circular nets, with corresponding quads lying in parallel planes.

\begin{definition}[Isothermic net]
\label{def:isothermicNet}
Let $(f,n)$ be a circular net. Then $(f,n)$ is a \emph{discrete isothermic net} if there exists a second circular net $(f^*, n)$  with the same Gau{\ss} map such that $A(f,f^*) = 0$. The net $(f^*,n)$ is unique (up to scaling and translation) and called the \emph{discrete Christoffel dual net of $(f, n)$}.
\end{definition}

We now state a few important properties of discrete isothermic nets that we will need \cite{Bobenko:2009tt}.

\begin{lemma}
Let $(f,n)$ be a discrete isothermic net with Christoffel dual $(f^*,n)$. Then the following hold:\\
\indent 1. There exists real values $\beta$ per edge that coincide for opposite edges on each quad and the cross-ratio of every quad factorizes, i.e.,
\begin{equation}
\frac{(f_1 - f)(f_{12} - f_2)}{(f_{12} - f_1)(f_2 - f)} = \frac{\beta_1}{\beta_2},
\end{equation}
with $\beta_1 = \beta_{f_1 - f}$ and $\beta_2 = \beta_{f_2 - f}$ associated to shifts in the first and second lattice directions, respectivly. \\
\indent 2. Corresponding edges of $f$ and $f^*$ are parallel and satisfy:
\begin{equation}
\label{eq:dualEdges}
f^*_i - f^* = \beta_i \frac{f_i - f}{\|f_i - f\|^2} ~\mathrm{for} ~ i=1,2,
\end{equation}
while non-corresponding diagonals are parallel and satisfy:
\begin{equation}
\label{eq:dualDiagonals}
f^*_{12} - f^* = (\beta_2 - \beta_1) \frac{f_2 - f_1}{\|f_2 - f_1\|^2} ~\mathrm{and}~ f^*_{2} - f^*_1 = (\beta_2 - \beta_1) \frac{f_{12} - f}{\|f_{12} - f\|^2}.
\end{equation}
\end{lemma}

For the rest of this section we restrict the discussion to the special case of discrete isothermic nets whose immersion quads have cross-ratio minus one, in particular, we will assume that $\beta_1 = 1$ and $\beta_2 = -1$ for every quad. The reason for doing this is that if we think of the cross-ratio as a discrete analog of $f_x^2/f_y^2$, then cross-ratio minus one corresponds to $f_x^2 = - f_y^2$, the defining property of conformal maps; the more general notion of factorizing cross-ratio allows for reparametrizations of the parameter lines.

It is essential to emphasize that the restriction to cross-ratio minus one solely serves the purpose of simplifying the algebra. Every result that follows also holds with the more general definition, with the pre-factors $\beta_1$ and $\beta_2$ cropping up in the expected places.

The notion of a discrete holomorphic function just restricts the notion of discrete isothermicity to the plane.

\begin{definition}
Fix a quad graph $\G$. The complex function $g: \G \to \C$ is called a \emph{discrete holomorphic function} if every quad has cross-ratio minus one.
\end{definition}

We can now define the discrete Weierstrass representation by following the same procedure as in the smooth case.

\begin{definition}[Weierstrass representation of discrete isothermic minimal nets]
\label{def:discreteWeierstrass}
Let $g: \Z^2 \to \C$ be a discrete holomorphic function and consider the discrete isothermic net $(f,n)$ where:\\
\indent 1. the Gau{\ss} map $n$ is given by the stereographic projection of the holomorphic data $g$, i.e.,
\begin{eqnarray}
n &:=& \frac1{1+ |g|^2}(g + \bar g, \frac1i (g - \bar g), |g|^2 -1), \mathrm{~so~in~particular~}\\
\nonumber
n_i &:=& \frac1{1+ |g_i|^2}(g_i + \bar g_i, \frac1i (g_i - \bar g_i), |g_i|^2 -1) \mathrm{~for~} i = 1,2;
\end{eqnarray}\\
\indent 2. and $f$ is given (up to translations) as the discrete isothermic dual immersion of $n$, i.e., for a shift in either direction
\begin{eqnarray}
\nonumber
  &f_1 -f := \frac{n_1-n}{\|n_1 - n\|^2} = \Re\left( \frac1{2(g_1 - g)} (1 - g_1 g, i(1+g_1 g), g_1 + g)\right) \mathrm{and} & \\
  &f_2 -f := -\frac{n_2-n}{\|n_2 - n\|^2} = -\Re\left( \frac1{2(g_2 - g)} (1 - g_2 g, i(1+g_2 g), g_2 + g)\right). &
\end{eqnarray}
We call the arising net a \emph{discrete isothermic minimal net}.
\end{definition}

\begin{lemma}
Let $(f,n)$ be a discrete isothermic minimal net. Then it is a minimal edge-constraint net.
\begin{proof}
Discrete isothermic nets are circular nets, so they are edge-constraint nets and by construction $(f,n)$ is the Christoffel dual of $(n,n)$, so $A(f,n) = 0$.
\end{proof}
\end{lemma}

In the smooth setting the Weierstrass representation provides a way to compute the Gau{\ss} curvature explicitly and for isothermic surfaces is given by $\K = \frac{-4 |g'|^4}{(1+|g|^2)^4}$, see~\cite{Oprea:2000wr}.
This has a discrete analogue:
\begin{lemma}
Let $(f,n)$ be a discrete isothermic minimal net arising from a discrete holomorphic function $g: \Z^2 \to \C$. Then the Gau{\ss} curvature of a quad is given in terms of $g$ by:
\begin{equation}
\K = \frac{-4 (| g_{12} - g | | g_2 - g_1 |)^2}{(1 + |g|^2)(1+|g_1|^2)(1+|g_{12}|^2)(1+|g_2|^2)}.
\end{equation}
\begin{proof}
Again we use the diagonals as the discrete partial derivatives. By Equation (\ref{eq:dualDiagonals}) we have $A(f,f) = - \frac{4}{\| n_x \|^2 \| n_y \|^2} A(n,n)$. Recall that the chordal distance between the stereographic projection of two points $w$ and $z$ in $\C$ is 
\begin{equation}
\frac{2}{\sqrt{(1 + |z|^2)(1 + |w|^2)}} |z - w |.
\end{equation}
Applying this to $ \| n_x \|^2$ and $\| n_y \|^2$ explicitly recovers the result.
\end{proof}
\end{lemma}

The discrete Weierstrass representation naturally gives rise to an associated family. We think of the discrete Weierstrass representation in terms of the corresponding \emph{discrete complex Weierstrass vectors}
\begin{equation}
\omega_i = \frac1{2(g_i - g)} (1 - g_i g, i(1+g_i g), g_i + g)
\end{equation}
for each lattice direction $i = 1, 2$ so that the immersion edges can be concisely written: $f_1 - f = \Re \omega_1$ and $f_2 - f = - \Re \omega_2$.

\begin{definition}[Associated family of discrete isothermic minimal net]
\label{def:minimalAssocFamily}
Let $g: \Z^2 \to \C$ be a discrete holomorphic function, $n: \Z^2 \to \Stwo$ be the stereographic projection of $g$, $\omega_i$ be the discrete complex Weierstrass vectors of $g$, and $\lambda = e^{i \alpha}$ for $\alpha \in [0, 2\pi]$. Then the family of contact element nets $(f^\alpha, n)$ given by:
\begin{equation}
f_1^\alpha -f^\alpha := \Re (\lambda \omega_1) \mathrm{~and~}  f_2^\alpha -f^\alpha := -\Re (\lambda \omega_2) 
\end{equation}
is called the \emph{associated family} of the discrete isothermic minimal net $(f^0, n)$ and $\lambda$ is called the \emph{spectral parameter}.
\end{definition}

\begin{lemma} Let $(f^0, n)$ be a discrete isothermic minimal net. Any member $(f^\alpha, n)$ of its associated family is an edge-constraint net as well.
\begin{proof}
By formally extending the dot product to complex 3-vectors\footnote{By the \emph{complex formal dot product} we mean $(z_1, z_2, z_3) \cdot (w_1, w_2, w_3) := \sum_{i=1}^3 z_i w_i$; there is no conjugation in the second component.} and noting that the Gau{\ss} map $n \in \R^3$ is real valued we have for $i = 1,2$ that $\Re(\lambda \omega_i) \cdot n = \Re (\lambda \omega_i \cdot n)$ and $\Re(\lambda \omega_i) \cdot n_i = \Re (\lambda \omega_i \cdot n_i)$. We easily compute $\Re (\lambda \omega_1 \cdot n) = -\Re (\frac\lambda2)$ and $\Re (\lambda \omega_1 \cdot n_1) = \Re (\frac\lambda2)$. Similarly, $\Re (\lambda \omega_2 \cdot n) = \Re (\frac\lambda2)$ and $\Re (\lambda \omega_2 \cdot n_2) = -\Re (\frac\lambda2)$. So $(f^\alpha_i - f^\alpha) \cdot n = - (f^\alpha_i - f\alpha) \cdot n_i$.
\end{proof}
\end{lemma}

The previous lemma provides us with a way to interpret the rotation generated by the multiplication of $\lambda$ directly in $\R^3$: 

\begin{figure}[ht]
  \centering
  \includegraphics[width=0.6\linewidth]{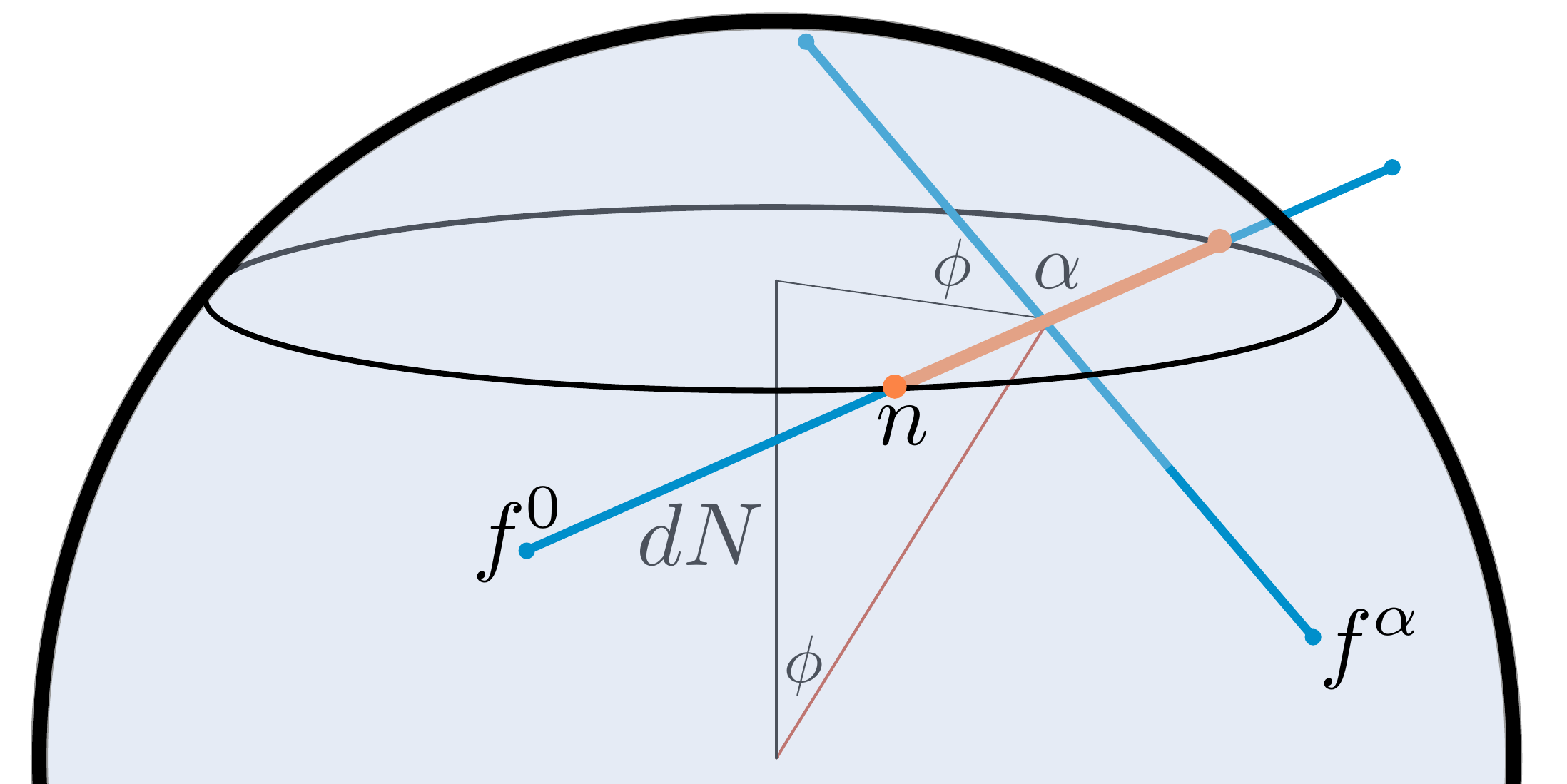}
  \caption{The geometry of an edge in the associated family (Equation (\ref{eq:geometricAssocFamily})) along with its original immersion edge (that is parallel to its Gau{\ss} map edge); the circle is the circumcircle of an adjacent Gau{\ss} map quadrilateral whose properties are needed to understand the geometry of immersion quads in the associated family (Lemma \ref{lem:projQuads}).}
  \label{fig:minAssocFamilyEdge}
\end{figure}

\begin{lemma}
Let $(f^\alpha, n)$ be the associated family of a discrete isothermic minimal net $(f^0, n)$. Then for each $\alpha \in [0,2\pi]$ and $i = 1,2$ we have:
\begin{equation}\label{eq:geometricAssocFamily}
  f^\alpha_i - f^\alpha = (-1)^{(i-1)} \| f^0_i - f^0 \|^2 (\cos \alpha (n_i - n) - \sin \alpha (n_i \times n)).
\end{equation}
\textup{In other words $f^\alpha_i - f^\alpha$ is given by $f^0_i - f^0$ rotated in the plane perpendicular to $(n_i + n)$ by angle $\alpha$, as shown in Figure \ref{fig:minAssocFamilyEdge}.}
\end{lemma}
\begin{proof}
We show the result for the first lattice direction $i = 1$, the other lattice direction follows similarly.

For every $\alpha \in [0, 2\pi]$ the immersion edge $f^\alpha_1 - f^\alpha$ is perpendicular to $(n_1 + n)/2$, so it is a linear combination of $(n_1 - n)$ and $(n_1 \times n)$. Furthermore, in terms of the complex discrete Weierstass vector, this immersion edge is the linear combination $\Re(\lambda \omega_1) = \cos\alpha \Re \omega_1 - \sin\alpha \Im \omega_1$. Since $(f^0,n)$ and $(n,n)$ are dual discrete isothermic nets we immediately have $\Re \omega_1 = \| f^0_1 - f^0\| ^2 (n_1 - n)$, so we only have to show $ \Im \omega_1  = \| f^0_1 - f^0\| ^2 (n_1 \times n)$.

A simple computation yields that the formal complex 3-vector dot product $\omega_1 \cdot \omega_1$ is real and equal to $1/4$. In particular, this implies that the real dot product $\Re \omega_1 \cdot \Im \omega_1 = 0$ and that $\| \Im \omega_1 \| ^2 = \| \Re \omega_1 \|^2 - 1/4$. Hence $\Im \omega_1$ is parallel to $n_1 \times n$ (since it is perpendicular to both $(n_1 - n)$ and $(n_1 + n)$). Now, since the Gau{\ss} map vectors are unit length we have $\|n_1 \times n \|^2 = \| \frac{n_1 + n}{2} \|^2 \| n_1 - n \|^2$, which we use to conclude
\begin{equation}
\| \Im \omega_1 \| ^2 = \frac{1}{\| n_1 - n \|^2} - 1/4 = \frac{\| n_1 + n \|^2}{4 \| n_1 - n \|^2} = ( \| f^0_1 - f^0 \|^2 \| n_1 \times n \| ) ^2.
\end{equation}
\end{proof}

\begin{figure}[ht]
  \centering
  \includegraphics[width=0.35\hsize]{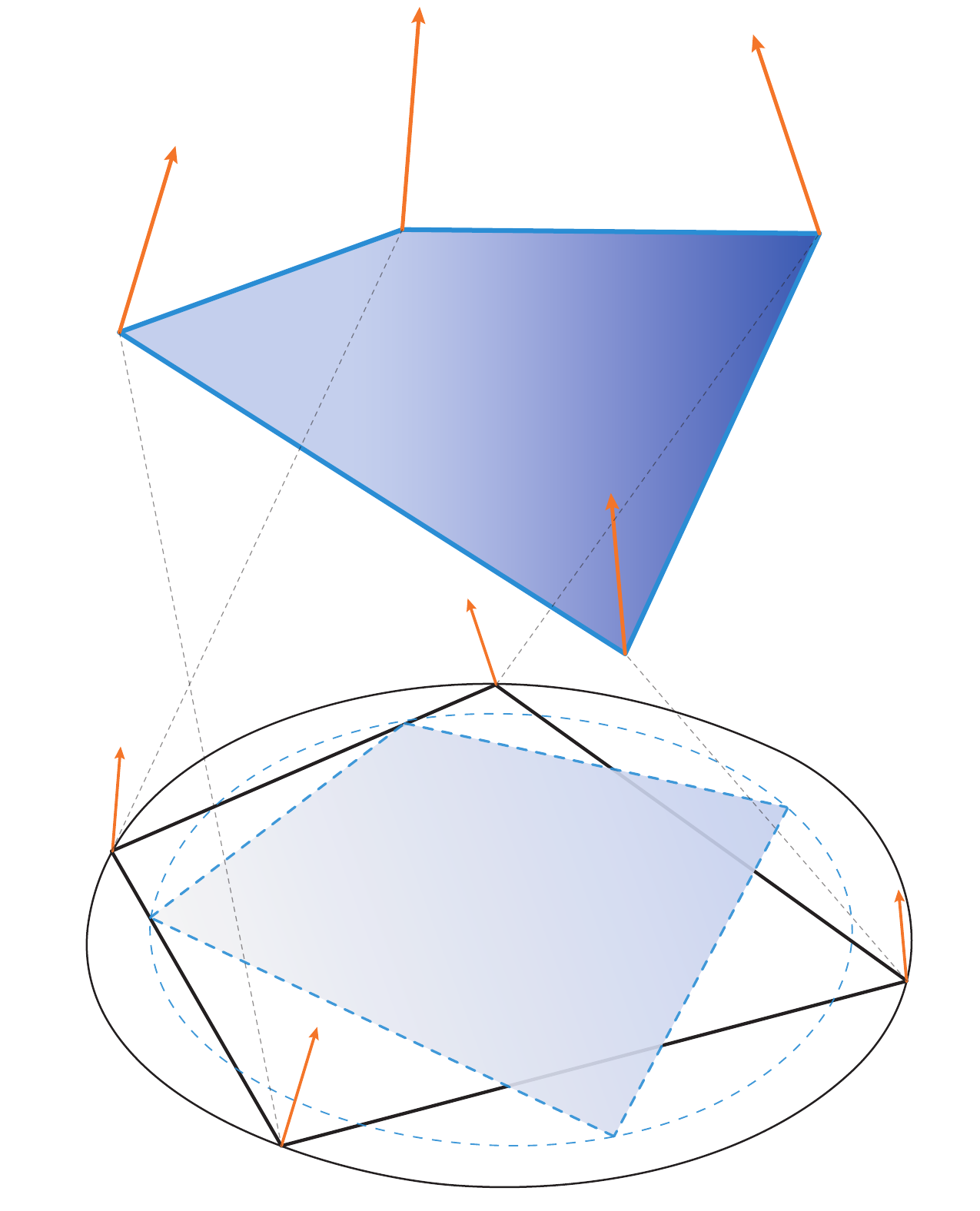}
  \caption{A non-planar quad (blue) from the associated family of a discrete isothermic minimal net and its projection, which is a rotated and scaled version of the original circular quad (black). Note that the Gau{\ss} map is constant.}
  \label{fig:projectedAssocMinimalQuad}
\end{figure}

This more geometric construction of the associated family highlights the following important relationship (shown in Figure \ref{fig:projectedAssocMinimalQuad}), which will lead to vanishing mean curvature.

\begin{lemma}[Quad geometry of the associated family]
\label{lem:projQuads}
Every immersion quad of a member of the associated family $(f^\alpha, n)$ when projected into its corresponding Gau{\ss} map plane is a scaled and rotated version of its corresponding circular immersion quad of $(f^0, n)$.
\begin{proof}
We show that an immersion edge of $(f^\alpha, n)$ when projected into one of its neighboring Gau{\ss} map planes is a scaled and rotated version of its corresponding immersion edge of $(f^0,n)$; the setup is given in Figure \ref{fig:minAssocFamilyEdge}. In particular, the scaling factor and rotation angle are independent of the lattice direction, so the result extends to the quads, proving the lemma.

What follows is identical for both lattice directions, so we work with the first one. The Gau{\ss} map quad is circular so the projection direction $N$ anchored at the origin passes through its circumcenter (at height $d$) and is normal to its plane; let $\pi$ be the projection into this plane. Furthermore, $N \perp (n_1 - n)$ so from Equation (\ref{eq:geometricAssocFamily}) we see that
\begin{equation}
\pi(f^\alpha_1 - f^\alpha) = (\| f^0_1 - f^0 \|)^2 (\cos \alpha (n_1 - n) - \sin\alpha \cos \phi (n_1 \times n)),
\end{equation}
where $\phi$ is the angle between $f^\alpha_1 - f^\alpha$ and the Gau{\ss} map plane. Observe that the angle between $N$ and $\frac{n_1 + n}{2}$ is also $\phi$, so we can calculate:
\begin{eqnarray}
\| \pi(f^\alpha_1 - f^\alpha) \|^2 &=& (\| f^0_1 - f^0 \| )^4 (\cos^2\alpha (l)^2 + \sin^2\alpha \cos^2\phi \|n_1 \times n \|^2) \nonumber \\
&=& (\| f^0_1 - f^0 \| )^2 ( \cos^2\alpha + \sin^2\alpha \cos^2 \phi \|\frac{n_1 + n}{2}\|^2). \nonumber \\
&=& (\| f^0_1 - f^0 \| )^2 (\cos^2\alpha + \sin^2\alpha~ d^2).
\end{eqnarray}
Therefore, the projected edge length is the original edge length scaled by a factor $\varsigma$ that is independent of the lattice direction and rotated by angle $\theta$, i.e.,
\begin{equation}
\varsigma = \sqrt{\cos^2\alpha + \sin^2\alpha ~ d^2} ~~ \mathrm{and} ~~ \cos\theta = \frac{\cos\alpha}{\sqrt{\cos^2\alpha + \sin^2\alpha~d^2}}.
\end{equation}
\end{proof}
\end{lemma}

We can now prove the main result of this section.

\begin{theorem}[Minimality of the associated family]
All contact element nets $(f^\alpha, n)_{\alpha \in [0,2\pi]}$ in the associated family of a discrete isothermic minimal net $(f^0, n)$ are minimal edge-constraint nets.
\begin{proof}
Choose an arbitrary $\alpha \in [0,2\pi]$. Consider a single quad of $(f^\alpha, n)$ with projection direction $N$, projection map $\pi$ and partial derivatives given by the diagonals. By Lemma \ref{lem:projQuads} there exist a rotation by angle $\theta$ in the plane perpendicular to $N$ and a scaling factor $\varsigma$ (both depending on $\alpha$) that bring the original immersion quad of $f^0$ into the projected associated family quad of $f^\alpha$. Noticing that $\cos\alpha$ equals zero or one exactly when $\cos\theta$ is also zero or one, respectively, we write:

\begin{eqnarray}
\pi(f^\alpha_x) &=& \varsigma \|f^0_x\| (\cos\theta \frac{f^0_x}{\|f^0_x\|} + \sin\theta \frac{f^0_y}{\|f^0_y\|}) ~ \mathrm{and} \nonumber \\
\pi(f^\alpha_y) &=& \varsigma \|f^0_y\| (-\sin\theta \frac{f^0_x}{\|f^0_x\|} + \cos\theta \frac{f^0_y}{\|f^0_y\|}).
\end{eqnarray}

The original net $(f^0, n)$ is discrete isothermic so Equation (\ref{eq:dualDiagonals}) implies that: $\frac{f^0_x}{\|f^0_x\|} = -\frac{n_y}{\| n_y \|}$ and $\frac{f^0_y}{\|f^0_y\|} = -\frac{n_x}{\| n_x \|}$; and that $\| f^0_x \| \| n_y \| = \| n_x \| \| f^0_y \|= 2$. Therefore, we compute twice the mixed area and see that it vanishes. 
\begin{eqnarray}
2 A(f^\alpha, n) &=& \det(\pi(f^\alpha_x), n_y, N) + \det(n_x, \pi(f^\alpha_y), N) \nonumber \\
&=& -2 \varsigma (\det(\sin\theta \frac{n_x}{\|n_x\|}, \frac{n_y}{\|n_y\|}, N) + \det(\frac{n_x}{\|n_x\|}, -\sin\theta \frac{n_y}{\|n_y\|}, N)) \nonumber \\
&=& 0.
\end{eqnarray}
\end{proof}
\end{theorem}

In contrast to the smooth case, the Gau{\ss} curvature does not stay constant in the discrete associated family, but it changes in a controlled way:

\begin{theorem}[Gau{\ss} curvature of the associated family]
Let $(f^\alpha, n)$ be contact element net in the associated family of a discrete isothermic minimal net $(f^0, n)$ and let $\K^\alpha$ and $\K^0$ be their respective Gau{\ss} curvatures. Then
\begin{equation}
\K^\alpha = \frac{\K^0}{\varsigma^2},
\end{equation}
where $\varsigma^2 = \cos^2\alpha + \sin^2\alpha ~d^2$ is the square of the scaling factor as above and $d$ is the distance to the circumcenter of the corresponding Gau{\ss} map quad.

\medskip
\noindent \textup{Note that $\varsigma$ will approach one in the continuum limit.}
\begin{proof}
By Lemma \ref{lem:projQuads} we have $A(f^\alpha,f^\alpha) = \varsigma^2 A(f^0,f^0)$ and $A(n,n)$ is constant throughout the family.
\end{proof}
\end{theorem}

As in the smooth case we can define the \emph{conjugate discrete isothermic minimal net} of $(f^0,n)$ as $(f^\frac\pi2,n)$. The members of the associated family are linear combinations of the discrete isothermic net and its conjugate net, since $(f^\frac\pi2,n)$ arises as the imaginary part of the complex Weierstrass vectors. The conjugate net is known to be in asymptotic line parametrization. Discrete analogues of such parametrizations are known as \emph{A-nets} and were originally introduced by Sauer \cite{Sauer:1950ca} and Wunderlich \cite{Wunderlich:1951wc} to investigate surfaces of constant negative Gau{\ss} curvature (as we will do in Section \ref{sec:knets}).

\begin{definition} [Discrete asymptotic net]
An edge-constraint net is an \emph{A-net} if its immersion has planar vertex stars (i.e., if all immersion edges meeting at a vertex lie in a common plane) and the Gau{\ss} map is given by choosing unit normals to these planes.
\end{definition} 
This definition corresponds to the fact that the osculating planes of asymptotic lines are the tangential planes of the surface.

\begin{lemma}
Every A-net is an edge-constraint net.
\end{lemma}
\begin{proof}
  This follows directly from the definition.
\end{proof}
We close by showing that conjugate discrete isothermic minimal nets are indeed A-nets.

\begin{lemma}
The conjugate net $(f^{\frac \pi 2},n)$ of a discrete isothermic minimal net $(f^0, n)$ is an A-net.
\begin{proof}
From equation (\ref{eq:geometricAssocFamily}) the edges of the conjugate net satisfy:
\begin{equation}
f^\frac\pi2_1 - f^\frac\pi2 = - \| f^0_1 - f^0 \| (n_1 \times n) ~ \mathrm{and} ~ f^\frac\pi2_2 - f^\frac\pi2 = \| f^0_2 - f^0 \| (n_2 \times n).
\end{equation}
In particular, this means that all edges emanating from a generic vertex of $f^\frac\pi2$ are perpendicular to the corresponding Gau{ss} map $n$, which is the definition of an A-net.
\end{proof}
\end{lemma}

\subsection{Discrete constant mean curvature surfaces}
\label{sec:cmc}
We start with a general definition. 

\begin{definition}[Constant mean curvature edge-constraint net]
An edge-constraint net $(f,n)$ is said to have \emph{constant mean curvature} if every quad of the net has the same non-vanishing mean curvature $\Hc \in \R^*$.
\end{definition}

Like their simpler minimal cousins, smooth surfaces of constant mean curvature are often described in isothermic parametrizations. In such coordinates the Gau{\ss}-Codazzi equation is given in terms of the conformal metric parameter $u$ (defined by $ds^2 = e^u (dx^2 + dy^2)$), and reduces to the integrable elliptic sinh-Gordon equation,
\begin{equation}
u_{x y} = \sinh(u).
\end{equation}
Techniques of soliton theory have been immensely successful in explicitly constructing and classifying constant mean curvature surfaces (e.g., tori \cite{Pinkall:1989up,Bobenko:1991vd}) in the classical setting of Euclidean three-space and other space forms. The integrability condition of a suitably gauged frame can be identified with the Lax representation of the integrable equation \cite{Sym:1985kl, Bobenko:1994tv} which harnesses methods from soliton theory for geometry and allows for structure preserving discretizations \cite{Bobenko:2008tn}. Moreover, one can then explicitly describe the immersed surfaces in terms of the so-called Sym--Bobenko formula, which by construction simultaneously generates the \emph{associated family}.

Using this method Bobenko \& Pinkall \cite{Bobenko:1999us} defined discrete constant mean curvature surfaces as a subclass of discrete isothermic surfaces, just as they did for discrete minimal surfaces. For smooth constant mean curvature surfaces the DPW method \cite{Dorfmeister:1998wq} is a Weierstrass type method that allows the construction of all cmc surfaces from holomorphic / meromorphic data. A discrete version of this method giving rise to the same frame description as Bobenko \& Pinkall can be found in \cite{Hoffmann:1999vm}.

As with the minimal case, these discrete \emph{isothermic} surfaces arising from the frame description have previously been shown to have constant mean curvature \cite{Bobenko:2010eg}, but once again, the naturally arising associated family leaves the realm of the special isothermic parametrization (to more general conformal parametrizations in the smooth setting). Thus there has been no notion of discrete mean curvature through which this family could be geometrically understood. Again, we rectify this by showing that the original discrete isothermic net and its entire algebraically generated associated family are in fact constant mean curvature edge-constraint nets.

Since the curvature theory of edge-constraint nets satisfies the Steiner formula (Equation (\ref{eq:steinerFormula})), the linear Weingarten relationship and its corollary come for free by calculating the curvatures of an offset surface.
\begin{lemma}[Linear Weingarten relationship]
Let $(f,n)$ be an edge-constraint net. Then for any $t \in \R$ the edge-constraint net given by the offset $f^t = (f + t n, n)$ has curvatures 
\begin{equation}
\K_{f^t} = \frac{\K}{1 + 2 \Hc t + \K t^2} ~\mathrm{and}~\Hc_{f^t} = \frac{\Hc + \K t}{1 + 2 \Hc t + \K t^2}, \label{eq:offsetCurvatures}
\end{equation}
where $\K, \Hc$ are the Gau{\ss} and mean curvatures of the original surface $(f,n)$. If $(f,n)$ has constant mean curvature, then there exist $\alpha, \beta \in \R$ only depending on $t, \Hc$ such that 
\begin{equation}
\alpha \K_{f^t} + \beta \Hc_{f^t} = 1.
\end{equation}
In other words, for each offset net, $\alpha, \beta$ are constant on all quads of the net.

\begin{proof}
Choose $\alpha = - t (\frac1\Hc + t)$ and $\beta = \frac1\Hc + 2t$.
\end{proof}
\end{lemma}

\begin{figure}[ht]
  \centering
  \includegraphics[width=\hsize]{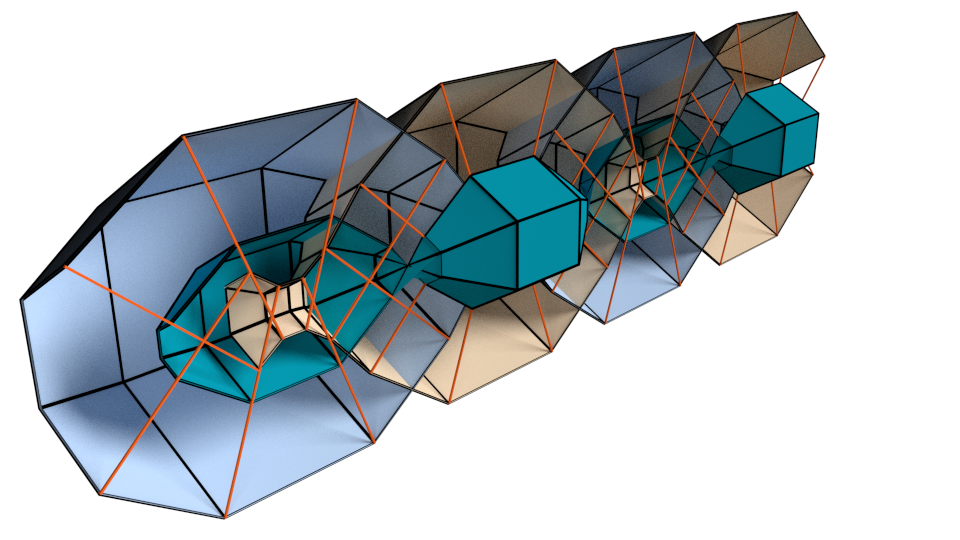}
  \caption{Two dual Delaunay nets with the positive Gau{\ss} curvature net inbetween. The normal lines connecting them are shown in orange.}
  \label{fig:parallelCMCnets}
\end{figure}

An important corollary is that constant mean curvature edge-constraint nets come in pairs, just like their smooth counterparts (see Figure~\ref{fig:parallelCMCnets}).

\begin{corollary}
Let $(f,n)$ be a constant mean curvature edge-constraint net with mean curvature $\Hc = -\frac1h$, for some $h \in \R^+$. Then the offset net $(f^*, n) := (f+h n, n)$ is also a constant mean curvature edge-constraint net with mean curvature $\Hc^* = -\mathcal{H}$ and the \emph{middle} edge-constraint net $(\hat f, n) := (f+\frac{h}{2} n, n)$ has constant positive Gau{\ss} curvature $4 \mathcal{H}^2$.
\end{corollary}
\begin{remark}
If $(f,n)$ is a discrete isothermic net of constant mean curvature $\Hc = -\frac1h$ then the offset net $(f+h n, n)$ is in fact the discrete Christoffel dual isothermic net (Definition \ref{def:isothermicNet}).
\end{remark}

For simplicity for the rest of our discussion we rescale to $\Hc = -1$.

\begin{lemma}
Let $(f,n)$ be an edge-constraint net with unit offset $(f^*,n) := (f+n,n)$. Consider a single quad, then
\begin{equation}
\label{eq:cmcVanishingMixedArea}
A(f,f^*) = 0 \iff \Hc = A(f,n)/A(f,f) = -1.
\end{equation}
\end{lemma}

In other words, vanishing mixed area between a net and its unit offset for every quad is equivalent to both nets having constant mean curvature. The condition
\begin{equation}
\label{eq:cmcogDefinition}
A(f,f^*) = \det(f_{12} - f, f_2^* - f_1^*, N) + \det(f_{12}^* - f^*, f_2 - f_1, N) = 0
\end{equation}
can be understood geometrically as the vanishing sum of the (projected) areas of the curves formed by $f, f_1^*, f_{12}, f_2^*$ and $f^*, f_1, f_{12}^*, f_2$ which we denote $g$ and $g^*$, respectively:
\begin{eqnarray}
&g = f, g_1 = f_1^*, g_{12} = f_{12},~\mathrm{and}~g_2 = f_2^*;&~\mathrm{and} \nonumber \\
&g^* = f^*, g_1^* = f_1, g_{12}^* = f_{12}^*,~\mathrm{and}~g_2^* = f_2.& \label{eq:combCubes}
\end{eqnarray}

Therefore, to prove that an edge-constraint net has constant mean curvature we switch between the two combinatorial cubes $C_f$ and $C_g$ formed by $f, f^*$ and $g,g^*$, respectively. They share the same vertex set but the edges of one are the diagonals of the other.

By showing that the algebraically generated associated family of discrete constant mean curvature nets of Bobenko and Pinkall \cite{Bobenko:1999us} are constant mean curvature edge-constraint nets, we find that their $C_g$ cubes are built from skew parallelograms, yielding an unexpected connection to the 3D compatibility cube for discrete curves from the theory of integrable systems \cite{Hoffmann:2008ub,TabachnikovBicycleMap}.
We now briefly recapitulate the moving frame description of these nets.

We identify Euclidean three space $\R^3$ with the imaginary part of the quaternions $\Quat = \spann{\{\qOne, \qi, \qj, \qk\}}$, i.e., $\R^3 \cong \spann{\{\qi, \qj, \qk\}} \cong \spann{\{ -i \sigma_1, -i \sigma_2, -i \sigma_3\}}$, where
\begin{equation}
\label{eq:PauliMatrices}
\sigma_1 = \mattwo0110, ~ \sigma_2 = \mattwo0{-i}{i}0, ~ \sigma_3 = \mattwo100{-1},
\end{equation}
are the 2x2 complex \emph{Pauli matrices} generating the Lie algebra su(2). Using this representation one describes the surface via a moving frame $\Phi \in \Im\Quat$ which rotates the orthonormal frame of $\R^3$ into the surface tangent plane and normal vector. Specifically, conjugation by a quaternion corresponds to a rotation and the frame encodes the Gau{\ss} map directly by rotating $\qk$, i.e., $n = - i \Phi^{-1} \sigma_3 \Phi = \Phi^{-1} \qk \Phi$. These frame descriptions are the natural language of integrable systems related to surface theory both smooth and discrete \cite{Bobenko:1994tv,Bobenko:1999us,Bobenko:2008tn}.

\begin{figure}[ht]
  \centering
  \includegraphics[width=0.3\hsize]{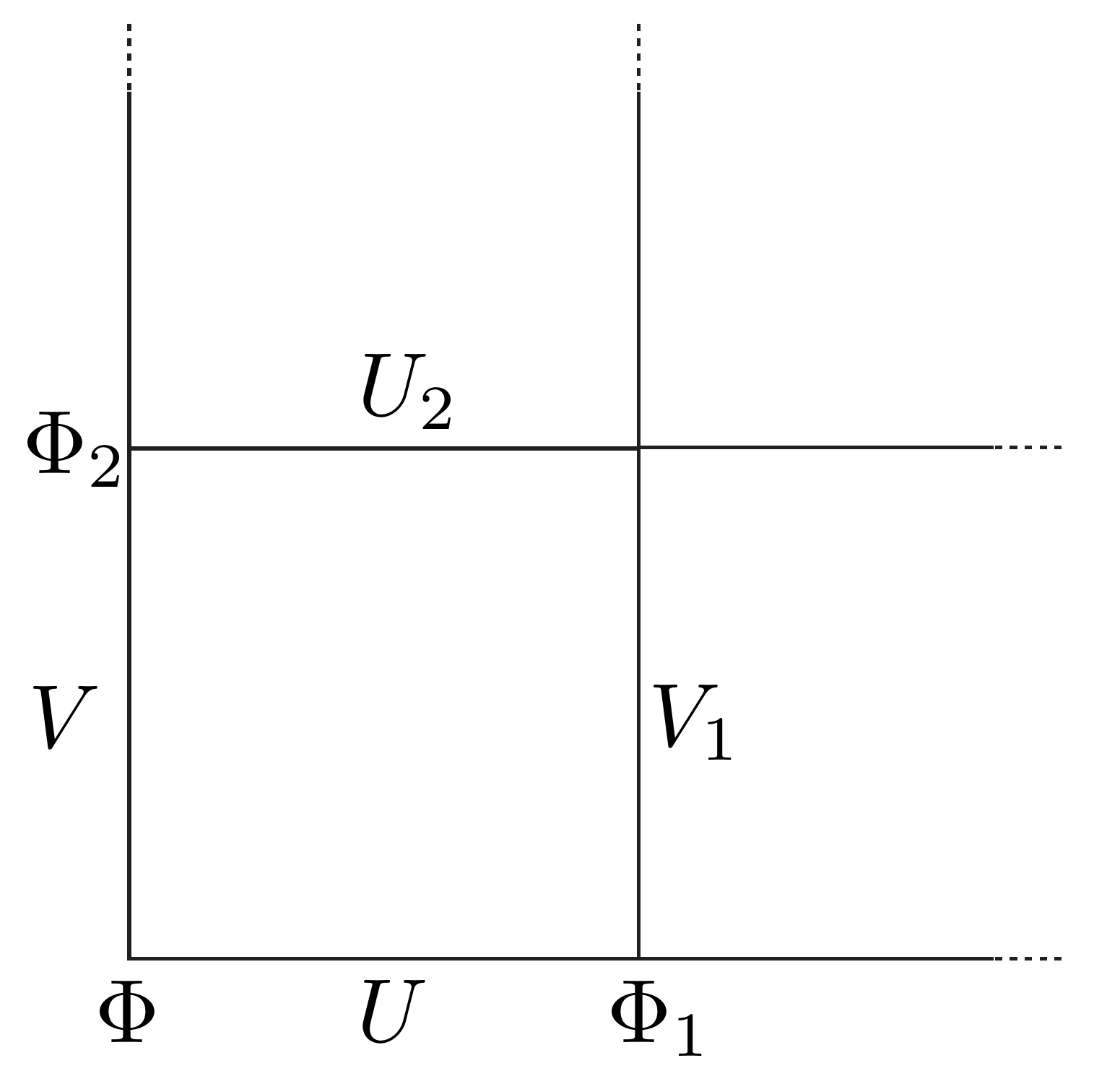}
  \caption{Shift notation for the frame description. The frame of a net $\Phi$ lives on vertices, while the corresponding \emph{Lax matrices} $U$ and $V$ live on edges; for both the frame and Lax matrices one can understand shifts in either lattice direction.}
  \label{fig:cmcFrameNotation.pdf}
\end{figure}

With the notation of Figure \ref{fig:cmcFrameNotation.pdf}, we introduce the frame $\Phi$ of interest \cite{Bobenko:1999us} by initially setting $\Phi$ equal to the identity and then defining the vertex shifts
\begin{equation}
\Phi_1 := U \Phi ~ \mathrm{and} ~ \Phi_2 := V \Phi,
\end{equation}
where 
\begin{equation}
  \label{eq:cmcFrame}
  U := \mattwo{a}{-\lambda u - \frac1{\lambda u}}
  {\frac{u}{\lambda} +
    \frac{\lambda}{u}}{\bar a} ~\mathrm{and}~
    V :=\mattwo{b}{-i \lambda v + \frac{i}{\lambda v}}
  {i\frac{\lambda }{v} -
    \frac{i}{\lambda}v}{\bar b}
\end{equation}
are the \emph{Lax matrices} with spectral parameter $\lambda = e^{i \alpha}$ for $\alpha \in [0,2\pi]$, $a, b$ complex valued functions living on vertices (and $\bar a, \bar b$ their complex conjugates), and $u,v$ positive real valued functions living on vertices. To guarantee that each quad closes, i.e., $\Phi_{12} = \Phi_{21}$, the Lax matrices and their shifts $V_1$ and $U_2$ must satisfy the compatibility condition (with determinants splitting evenly):

\begin{equation}
  V_1U = U_2V, ~\mathrm{with}~ \det(U) = \det(U_2) ~\mathrm{and}~ \det(V) = \det(V_1). \label{eq:zeroCurvatureCond}
\end{equation}
For every value of the spectral parameter the net is then generated by taking the imaginary part\footnote{Unlike in \cite{Bobenko:1999us} we purposefully \emph{do not} normalize the transport matrices $U$ and $V$ to have determinant 1. Thus $\Phi$ is not in $SU(2)$, necessitating taking the imaginary part.} of the Sym--Bobenko formula \cite{Bobenko:1999us} (an example is shown in Figure \ref{fig:SmythNets}),
\begin{eqnarray}
	n &:=& \Phi^{-1} \qk \Phi \nonumber \\
	\qf &:=& \left(-\Phi^{-1}\frac{d}{d\alpha}\Phi\vert_{\lambda = e^{i\alpha}} + \frac{1}{2}\Phi^{-1}\qk\Phi \right), \label{eq:symBobenko} \\ 
 	f &:=& \Im \qf.\nonumber
\end{eqnarray} 

\begin{figure}[ht]
  \centering
  \includegraphics[width=.33\hsize]{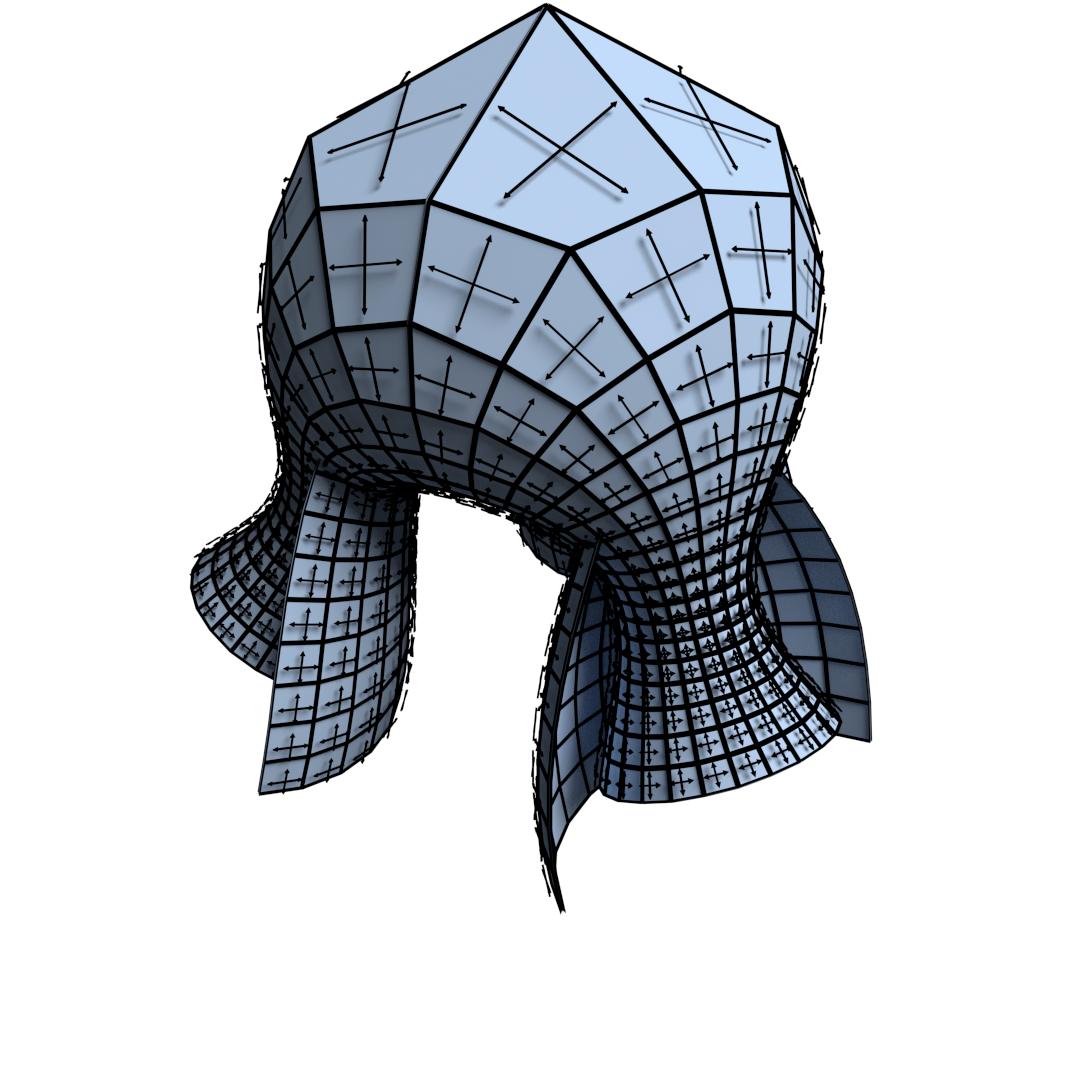}\includegraphics[width=.33\hsize]{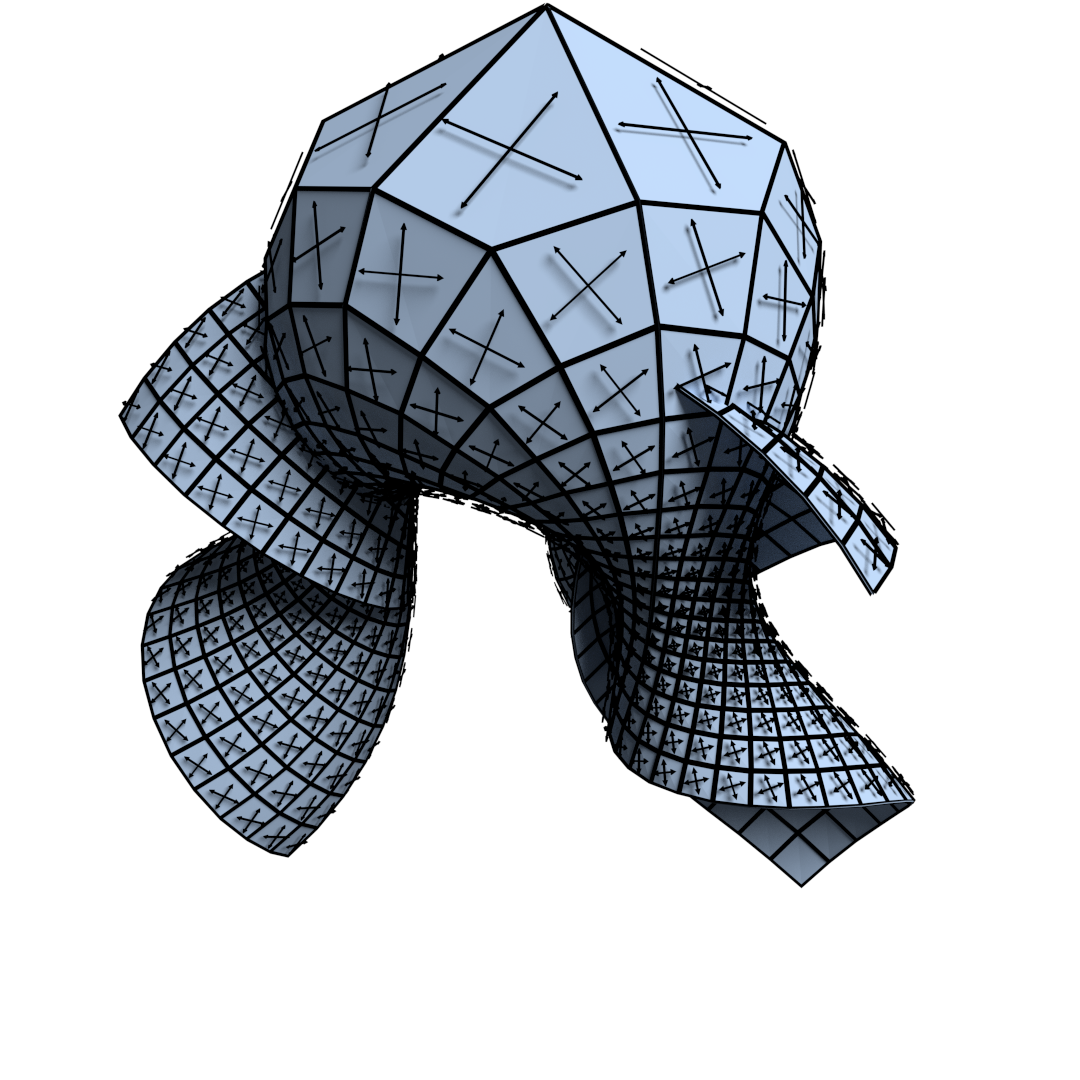}\includegraphics[width=.33\hsize]{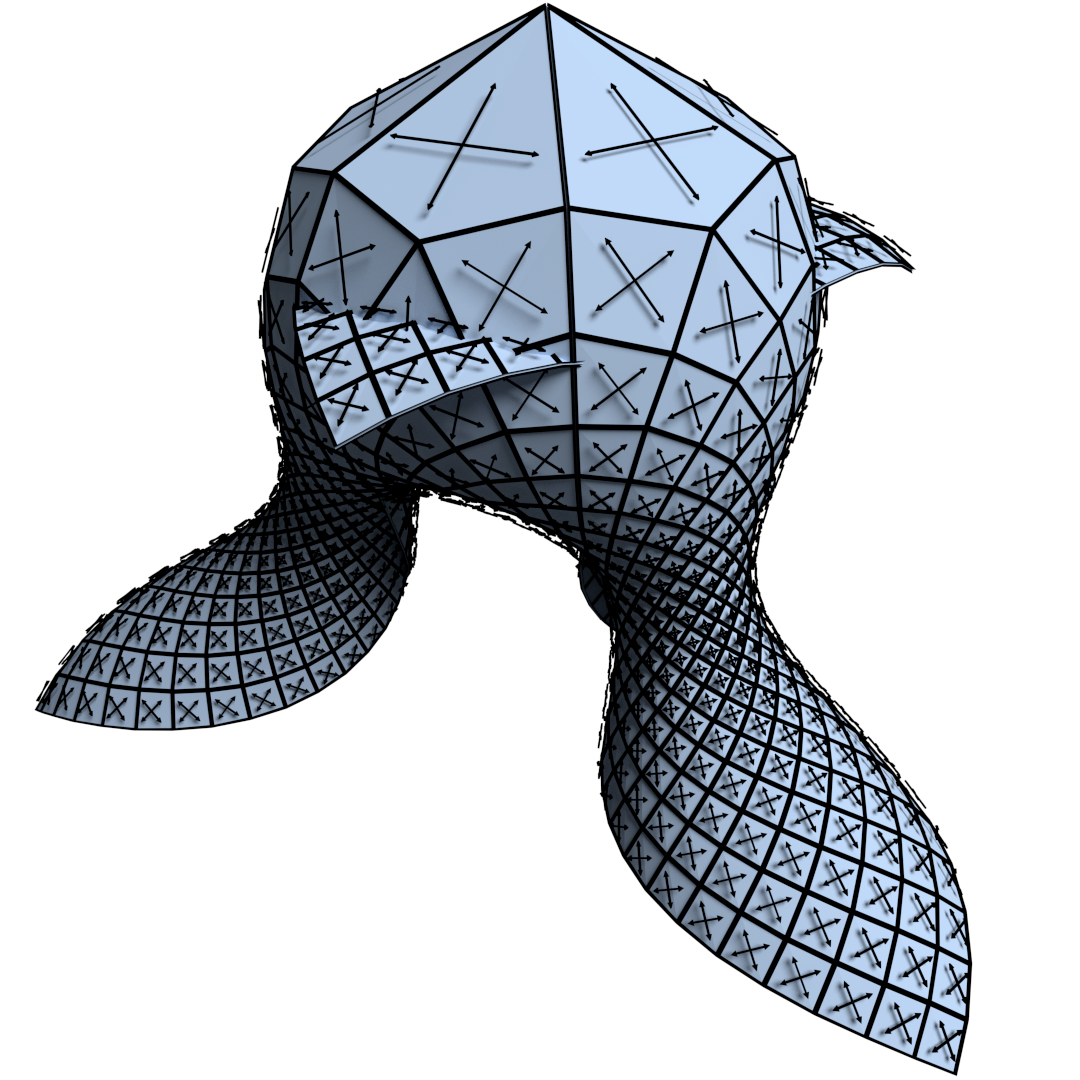}
  \caption{Three members of the associated family of a discrete constant mean curvature edge-constraint net Smyth surface, together with their curvature line fields. For smooth Smyth surfaces the associated family is known to be a reparametrization. Note how the curvature line fields keep their directions in the family.}
  \label{fig:SmythNets}
\end{figure}

\begin{definition}[Associated family]
\label{def:assocFamCMC}
Let $(f,n) := (f^\alpha,n^\alpha)$ be a net generated from the frame $\Phi$ (with spectral parameter $\lambda = e^{i \alpha}$) given by Equations (\ref{eq:cmcFrame}), (\ref{eq:zeroCurvatureCond}), and (\ref{eq:symBobenko}). We call $(f,n)$ a \emph{member of the associated family of a discrete isothermic constant mean curvature net}.
\end{definition}

\begin{lemma}
\label{lem:assocFamCMCEdgeConst}
Every member of the associated family of a discrete isothermic constant mean curvature net is an edge-constraint net. Furthermore, for $i = 1,2$ the edge-constraint is expressed in the equations
\begin{equation}
\label{eq:quatEdgeConst}
n_i  = -(\qf_i - \qf) n (\qf_i - \qf)^{-1}.
\end{equation}
\begin{proof} We defer the proof to Section \ref{sec:LaxPairedgeConstraint} where a general discussion of edge-constraint nets arising from Lax pairs is provided.
\end{proof}
\end{lemma}

In general, the edge-constraint can be understood along a shift in either lattice direction $i=1,2$ as first negating $n$ and then rotating it along the edge $f_i - f$ to find $n_i$, which, when written quaternionically, gives rise to the following definition.

\begin{definition}[Normal transport quaternions]
\label{def:normalTransport}
Consider a quad from an edge-constraint net $(f,n)$. The quaternions given by $\phi := \tau + (f_1 - f), \tau \in \R$ and $\psi = \eta + (f_2 - f), \eta \in \R$ such that
\begin{equation}
n_1 = - \phi^{-1} n \phi ~\mathrm{and}~ n_2 = - \psi^{-1} n \psi,
\end{equation}
are called \emph{normal transport} quaternions.\footnote{Although inverses naturally arise on the right (Equation (\ref{eq:quatEdgeConst})) from the Sym--Bobenko formula, we prefer to define normal transports with inverses on the left; this simply corresponds to an opposite sign convention for the real part of the normal transport.}
\end{definition}

This perspective yields insight into the geometry of the cubes $C_f$ and $C_g$ for an arbitrary edge-constraint net.

\begin{lemma}[Edge-constraint as a skew parallelogram]
Let $(f,n)$ be an edge-constraint net with offset net $(f^*,n) = (f+n,n)$. For each quad, consider the combinatorial cubes $C_f$ and $C_g$ formed by it and its offset (as given in Equation (\ref{eq:combCubes})). Then the four sides of $C_f$ and $C_g$ are skew trapezoids and skew parallelograms, respectively.
\end{lemma}
Moreover, we will see that if $(f,n)$ is a member of the associated family of a discrete isothermic constant mean curvature net then for every quad we have the following three facts, that together imply that $(f,n)$ has constant mean curvature: (i) the top and bottom of $C_g$ are also parallelograms; (ii) all six sides are parallelograms of the same "folding parameter", so $C_g$ forms an "equally-folded parallelogram cube"; and (iii) every equally-folded parallelogram cube has vanishing (projected) mixed area between its top and bottom.

Let $g,g_1,g_{12},g_2$ be a skew parallelogram built from the edge lengths $\ell_1$ and $\ell_2$. It is straightforward to see that the dihedral angles $\delta_1$ and $\delta_2$ (measured between $0$ and $\frac\pi2$) along the diagonals of this skew parallelogram (understood as edges of the enclosing tetrahedron) satisfy $\frac{\sin\delta_1}{\ell_1} = \frac{\sin\delta_2}{\ell_2}$.

\begin{definition}
The \emph{folding parameter} of a skew parallelogram with the above notation is defined as 
\begin{equation}
\sigma := \frac{\sin\delta_1}{\ell_1} = \frac{\sin\delta_2}{\ell_2}.
\end{equation}
\end{definition} 

\begin{lemma}
\label{lem:buildSkewParallel}
Every skew parallelogram can be written in terms of two edges $g_1-g$ and $g_2 - g$ with lengths $\ell_1$ and $\ell_2$, respectively, and a folding parameter $\sigma$:
\begin{eqnarray}
\label{eq:skewParallelogramRotation}
&g_{12} - g_2 = (\rho \qOne + (g_2 - g_1)) (g_1 - g) (\rho \qOne + (g_2 - g_1))^{-1}, \nonumber \\
\label{eq:realPartFoldingParameter}
&\mathrm{where}~\rho = \frac1\sigma \left(\sqrt{1 - \sigma^2 \ell_1^2} - \sqrt{1 - \sigma^2 \ell_2^2}\right).
\end{eqnarray}
\begin{proof}
The real part $\rho$ is the same as $\Re \nu$ in Equation 3.15 of \cite{Hoffmann:2008ub}, with $k = \tan\frac{\delta_1}{2}\cot\frac{\delta_2}{2}$ and $s = \ell_1$. Using Jacobi Elliptic functions one can rewrite this expression to find the above equation.
\end{proof}
\end{lemma}

This construction can be extended to three edges and a fixed folding parameter, yielding the known combinatorial 3D compatibility cube of skew parallelograms \cite{Hoffmann:2008ub,TabachnikovBicycleMap}:

\begin{theorem} [Darboux transform for parallelograms]
Let $g$ be a skew parallelogram with edge lengths $\ell_1, \ell_2$ and folding parameter $\sigma$. For every initial vector $\tilde n \in \R^3$ there exists a unique skew parallelogram $g^*$ at constant distance $\| \tilde n \|$ from $g$ such that:\\
\indent 1. $g^*$ also has edge lengths $\ell_1, \ell_2$ and folding parameter $\sigma$; and\\
\indent 2. every face of the combinatorial cube $C_g$ formed by $g$ and $g^*$ is a skew parallelogram of folding parameter $\sigma$.\\
We call this object an \emph{equally-folded parallelogram cube}.\footnote{Instead of fixing the folding parameter $\sigma$ one can also hold the real part $\rho$ constant; this is also 3D compatible as shown in \cite{Pinkall:2007uy}.}
\begin{proof}
Recall that $g$ itself can be generated from the two vectors $w_1 := g_1 - g$ and $w_2 := g_2 - g$ and the folding parameter $\sigma$. Therefore, we can rephrase the theorem statement as: Given $w_1, w_2, \tilde n$ and the folding parameter $\sigma$, show that completing the skew parallelogram twice in every direction forms a closed combinatorial cube. This is precisely the Bianchi Permutability Theorem for a single edge in the Darboux (B\"acklund) transformation of a discrete arc-length parametrized curve, a proof of which is given in \cite{Hoffmann:2008ub}.
\end{proof}
\end{theorem}
We now come back to the viewpoint that one can switch between the combinatorial cubes $C_g$ and $C_f$ as introduced in Equation (\ref{eq:combCubes}).
\begin{theorem}
\label{thm:equallyFolded}
Consider an equally-folded parallelogram cube $C_g$ with bottom and top $g$ and $g^*$, respectively. Let $h \in \R$ be the constant distance between $g$ and $g^*$. Then the bottom and top quads of the corresponding cube $C_f$ with normals given by the vertical edges of $C_f$ are edge-constraint net quads with mean curvatures $-\frac1h$ and $\frac1h$, respectively.
\begin{proof}
Notice that $C_g$ can be constructed from three vectors and a folding parameter using the simplified equation for the real part of the rotation quaternions, Equation (\ref{eq:realPartFoldingParameter}).
Using the quaternionic description and $\hat G := ((g_{12}^* - g_{12}) - (g^* - g)) \times ( (g_2 - g_2^*) - (g_1 - g_1^*) )$ one finds that
\begin{equation}
\label{eq:gsumVanishing}
\det(g_{12}^* - g^*, g_2^* - g_1^*, \hat G) + \det(g_{12} - g, g_2 - g_1, \hat G) = 0.
\end{equation}
Using that $N \parallel \hat G$, this is equivalent to $A(f,f^*) = A(f, f + h n) = 0$. Therefore, the edge-constraint quads $(f,n)$ and $(f + h n, n)$ indeed have mean curvatures -$\frac{1}{h}$ and $\frac1h$.
\end{proof}
\end{theorem}
\begin{remark}
The top and bottom faces in the previous theorem can be exchanged for any pair of opposite faces (i.e., front and back or left and right). It turns out that the direction of $\hat G$ defined in the previous proof is independent of this choice (possibly up to sign). In other words, the quad tangent planes arising from every pair of opposite faces coincide.
\end{remark}

\begin{theorem}
Let $(f,n)$ be a member of the associated family of a discrete isothermic constant mean curvature net (Definition \ref{def:assocFamCMC}) with spectral parameter $\lambda = e^{i \alpha}$. Then $(f,n)$ is a constant mean curvature edge-constraint net.
\begin{proof}
By Lemma \ref{lem:assocFamCMCEdgeConst} $(f,n)$ is an edge-constraint net. Consider the unit offset net $(f^*,n) = (f+n, n)$. For every quad we show that the corresponding combinatorial cube $C_g$ is an equally-folded parallelogram cube; the result then follows from Theorem \ref{thm:equallyFolded}.

We naturally extend the quaternionic description of $f = \Im \qf$ and $f^* = \Im (\qf + n)$ to $g := \Im \qg$ and $g^* := \Im \qg^*$, e.g., $\qg_1^* - \qg = \qf_1 - \qf$. The non-unit edges of the parallelograms of the front and left sides of $C_g$ are found to have squared lengths:
\begin{eqnarray}
&\ell_1^2 := \| g_1 - g \|^2 = 1 - 4(\frac{\cos(2\alpha)}{\det(U)} + \frac{\sin^2(2\alpha)}{\det(U)^2})&~\mathrm{and}\nonumber \\
\label{eq:cmcEdgeLengths}
&\ell_2^2 := \| g_2 - g \|^2 = 1 + 4(\frac{\cos(2\alpha)}{\det(V)} - \frac{\sin^2(2\alpha)}{\det(V)^2}).&
\end{eqnarray}
Recall that by assumption $\det(U) = \det(U_2)$ and $\det(V) = \det(V_1)$, so the back and right sides of $C_g$ are also parallelograms with non-unit edge lengths $\ell_1$ and $\ell_2$, respectively. Therefore, the top and bottom of $C_g$ are also parallelograms both built from the edge lengths $\ell_1$ and $\ell_2$.

The transports (in the sense of Lemma \ref{lem:buildSkewParallel}) that yield the front, left, and top skew parallelograms of $C_g$ are:
\begin{eqnarray}
\label{eq:gcubeRotations}
&- n_1 = (\qg_1^* - \qg) n (\qg_1^* - \qg)^{-1},& \nonumber \\
&- n_2 = (\qg_2^* - \qg) n (\qg_2^* - \qg)^{-1},& ~\mathrm{and} \\
&\qg_{12} - \qg_2 = (\qg_2 - \qg_1)(\qg_1 - \qg)(\qg_2 - \qg_1)^{-1}.& \nonumber
\end{eqnarray}
The real parts arising from the Sym--Bobenko formula for corresponding edges, e.g., $\Re(\qf_1 - \qf) = \foh\tr(U^{-1}\frac{\partial}{\partial\alpha}U)$, can be computed directly using the derivative of the determinant:
\begin{equation}
\frac{\partial}{\partial \alpha}\det(U) = \det(U) \tr(U^{-1}\frac{\partial}{\partial\alpha}U).
\end{equation}
Applying this formula to the "double transport", e.g. $V_1U$, shows that the sum of the real parts of $\qf_1 - \qf$ and $\qf_{12} - \qf_1$ is in fact the real part of the diagonal transport furnished by $\qf_{12} - \qf$. 
We therefore find:
\begin{eqnarray}
&\Re(\qg_1^* - \qg) = \frac{2 \sin(2 \alpha)}{\det(U)}, &\nonumber\\
\label{eq:cmcRealParts}
&\Re(\qg_2^* - \qg) = -\frac{2 \sin(2 \alpha)}{\det(V)},~ \mathrm{and} \\
&\Re(\qg_2 - \qg_1) = \Re(\qg_2 - \qg) - \Re(\qg_1 - \qg) = -\frac{(\det(U) + \det(V))\sin(2\alpha)}{\det(U)\det(V)}. \nonumber
\end{eqnarray} 
Plugging the real parts from Equations (\ref{eq:cmcRealParts}) with the edge lengths from Equations (\ref{eq:cmcEdgeLengths}) into the Equation (\ref{eq:realPartFoldingParameter}) yields the folding parameter $\sigma = \sin(2\alpha)$ in every instance.

Therefore $C_g$ is an equally-folded parallelogram cube. 
\end{proof}
\end{theorem}

\subsection{Discrete constant negative Gau{\ss} curvature surfaces}
\label{sec:knets}
We start with a general definition motivated by the smooth setting.

\begin{definition}[Constant negative curvature edge-constraint nets]
An edge-constraint net $(f,n)$ is said to have \emph{constant negative (Gau{\ss}) curvature} if every quad of the net has the same negative non-vanishing Gau{\ss} curvature $\K \in \R^-$.
\end{definition}

\begin{figure}[ht]
  \centering
  \includegraphics[width=\hsize]{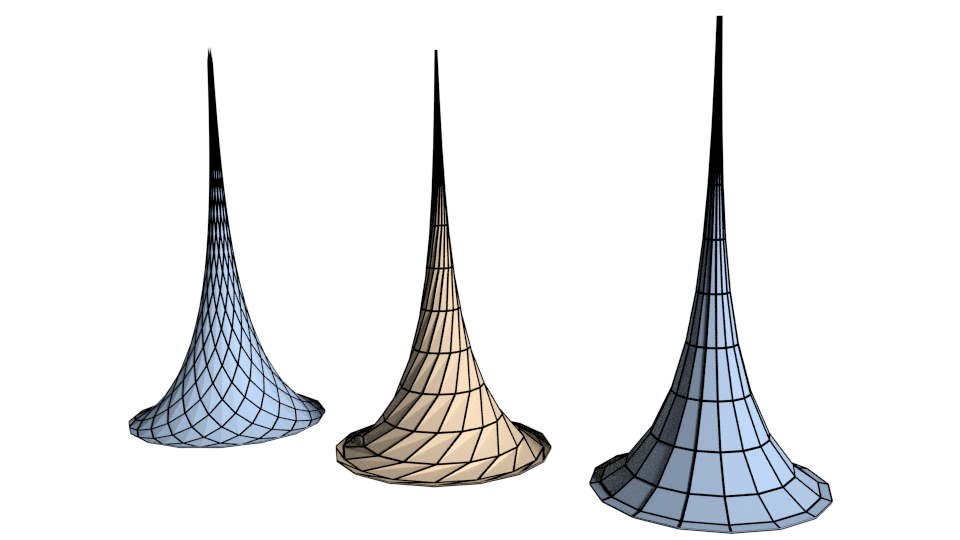}
  \caption{Three pseudospheres of revolution which are constant negative curvature edge-constraint nets. Left: Discrete asymptotic line parametrization (K-net). Middle: One discrete asymptotic line and one discrete curvature line parametrization. Right: Discrete curvature line parametrization (cK-net).}
  \label{fig:threePseudospheres}
\end{figure}

A surface parametrized by asymptotic lines has constant negative Gau{\ss} curvature if and only if the asymptotic lines form a \emph{Chebyshev net}, i.e., the parameter lines are parallel transports of each other in the sense of Levi-Civita. Thus, the directional derivative along each parameter line depends only on a single variable, so its integral curve exhibits a constant speed parametrization (possibly a different constant for each parameter line); Chebyshev nets are made up of "infinitesimal parallelograms" with side lengths $a,b \in \R^+$. Furthermore, the Gau{\ss} map of the asymptotic lines also form a Chebyshev net on the sphere. In these coordinates, the Gau{\ss}-Codazzi equation reduces to the well known sine-Gordon equation in the angle between the asymptote lines $u(x,y)$:
\begin{equation}
-\K a b \sin u(x,y) =  u_{xy}(x,y).
\end{equation}
This equation is invariant to the transformation $a \to \lambda a$ and $b \to \lambda^{-1} b$ for all $\lambda = e^t$ with $t \in \R$; varying $\lambda$ generates the \emph{associated family} of a constant negative curvature asymptotic line parameterized surface, where the angles between the asymptotic lines are invariant.

From the above characterization Sauer \cite{Sauer:1950ca,Sauer:1970tk} defined the following discrete analogue.
\begin{definition}[K-nets]
An edge-constraint net $(f,n): D \subset \Z^2 \to \R^3 \times \Stwo$ is a \emph{K-net} if it is an A-net and there exists two lengths $a, b > 0$ such that every immersion quad is a skew parallelogram (\emph{Chebyshev quad}) with edge lengths $a$ and $b$. \footnote{For K-nets we restrict to regular combinatorics because having more than two asymptotic lines meet at a point is incompatible with having negative Gau{\ss} curvature.}
\end{definition}
\noindent Note that the Gau{\ss} map of a K-net is also built from Chebyshev quads.

The associated family for K-nets was defined geometrically by Wunderlich \cite{Wunderlich:1951wc} using a transformation that, like in the smooth setting, preserves the interior angles of $f$ while scaling its edges. These geometric constructions agree with an algebraic description (similar to that introduced for constant mean curvature nets in the previous section) in terms of a discrete sine-Gordon equation and its Lax pair that is then integrated via a Sym--Bobenko formula to construct the net \cite{Bobenko:1996ug,Hoffmann:1999wq}. However, due to the inherent non-planarity of the quads in these nets, understanding their curvatures has remained elusive. The goal of this section is to show (using the geometric constructions) that K-nets indeed have constant negative Gau{\ss} curvature as edge-constraint nets.

K-nets can be constructed (up to global rotation and scale) directly from Cauchy data for their Gau{\ss} map \cite{pinkall2008designing}. Explicitly, a fourth Gau{\ss} map point $n_{12}$ is determined by completing the skew parallelogram through three other points $n, n_1, n_2 \in \Stwo$: \footnote{In other words the Gau{\ss} map satisfies the discrete Moutard equation in $\Stwo$ \cite{Nimmo:1997dg}.}
\begin{equation}
\label{eq:knetGaussquad}
n_{12} := \frac{n \cdot (n_1 + n_2)}{1 + n_1 \cdot n_2} (n_1 + n_2) - n.
\end{equation}
The immersion $f$ is constructed from $n$ via
\begin{equation}
\label{eq:knetImmersion}
f_1 - f := n_1 \times n~\mathrm{and}~ f_2 - f := n \times n_2.
\end{equation}
When $n_i \cdot n = \cos\Delta_i$ the edge lengths of $f$ are given by $\sin\Delta_i$ for $i = 1,2$. Moreover, applying Napier's analogies to the spherical parallelogram formed by $n$ the two interior angles $\alpha, \beta \in (0, \pi)$ of a K-net immersion quad are related by $e^{i \alpha} = \frac{e^{i \beta} k - 1}{e^{i \beta} - k}$, where $k = \tan\frac{\Delta_1}{2}\tan\frac{\Delta_2}{2}$ \cite{Bobenko:1996ug}.

The associated family can now be described by a family of pairs of Gau{\ss} map angles $(\Delta_1(\lambda), \Delta_2(\lambda))$ from which the K-nets are explicitly constructed.

\begin{definition}[K-net associated family]
Consider a K-net $(f,n)$ with the above notation and let $\lambda = e^t$ for all $t \in \R$. We construct a new K-net from $\Delta_1(\lambda)$ and $\Delta_2(\lambda)$ defined by the transformation
\begin{equation}
\label{eq:deltaAssoc}
\tan\frac{\Delta_1(\lambda)}{2} := \lambda \tan\frac{\Delta_1}{2}, ~ \tan\frac{\Delta_2(\lambda)}{2} := \lambda^{-1} \tan\frac{\Delta_2}{2},
\end{equation}
\end{definition}
\noindent The interior angles of every quad are invariant to this transformation and the edge lengths transform as $\sin\Delta_1(\lambda)$ and $\sin\Delta_2(\lambda)$.

From this construction one can compute the Gau{\ss} curvature of every quad directly yielding our main theorem.

\begin{theorem}
Every K-net has constant negative Gau{\ss} curvature.
\begin{proof}
Direct computation using Equations (\ref{eq:knetGaussquad}) and (\ref{eq:knetImmersion}) yields (up to global scaling)
\begin{equation}
\K = \frac{-2}{(n_1 + n_2) \cdot n} = \frac{-2}{\cos{\Delta_1(\lambda)} + \cos{\Delta_2(\lambda)}}.
\end{equation}
Note that $\K$ does depend on $\Delta_1(\lambda)$ and $\Delta_2(\lambda)$ but for a fixed $\lambda \in \R^+$ both of these angles are constant (by definition) for every quad of the K-net \footnote{Some authors define K-nets in the weakest sense where the $\Delta_1(\lambda)$ and $\Delta_2(\lambda)$ are allowed to vary along the parameter lines as single variable functions \cite{Bobenko:2008tn}. While still edge-constraint nets, they obviously do not have constant negative Gau{\ss} curvature.}. To have the same negative constant for all members of the associated family one must globally scale by a value dependent on $\lambda$.
\end{proof}
\end{theorem}
\begin{remark}
Wunderlich gave an interpretation for curvatures of K-nets in the symmetric case of $\Delta_1=\Delta_2$ \cite{Wunderlich:1951wc}: He interprets the circles that touch pairs of incident triangles in opposite points in the quad symmetry planes as the curvature circles and shows that the product of their radii is constant (see Figure~\ref{fig:curvatureCircles}). This quantity is in fact the Gau{\ss} curvature as defined by Equation (\ref{eq:meanGaussCurvature}); the radii of the circles are the ratios of diagonals in the $f$ and $n$ quadrilaterals and since the diagonals are perpendicular in both quadrilaterals the product of their lengths is proportional to the (projected) area.
\end{remark}
\begin{figure}[ht]
  \centering
  \includegraphics[width=.5\hsize]{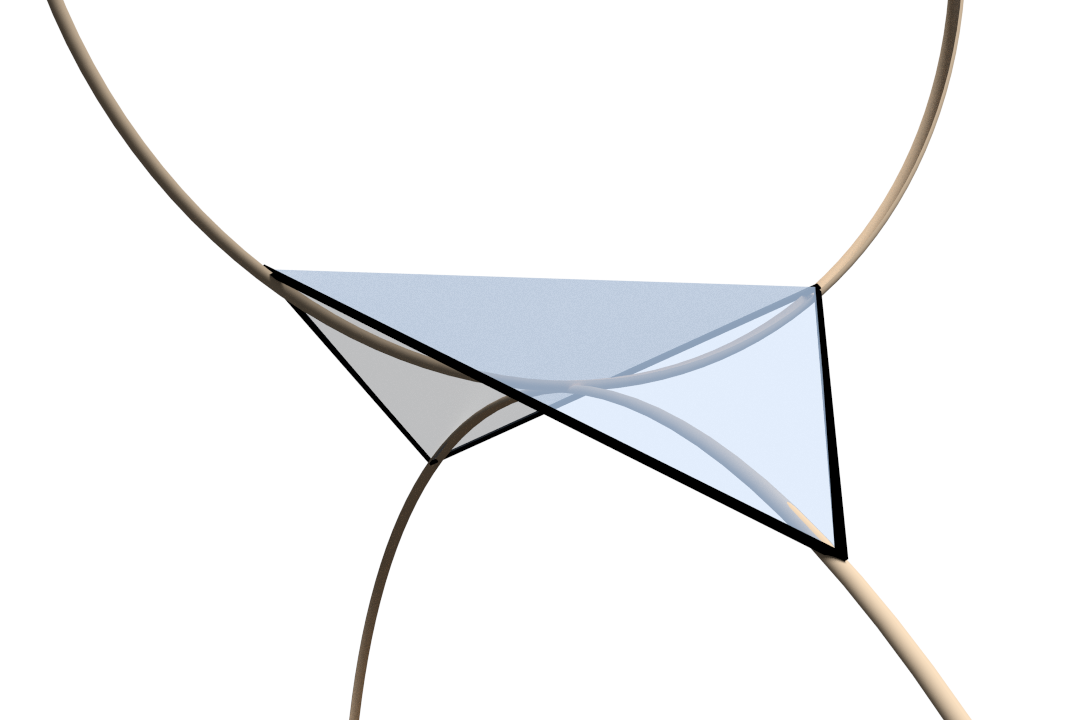}
  \caption{Wunderlich's curvature circles for a symmetric K-net quadrilateral.}
  \label{fig:curvatureCircles}
\end{figure}

Wunderlich's ideas stem from the fact that in the smooth setting the angular bisectors of the asymptotic lines are the curvature lines. Observe that for K-nets the diagonals of each quad do indeed satisfy the edge-constraint with the same normals, e.g., $f_{12} - f \perp n_{12} + n$ and for a K-net with all edge lengths equal we even have $n_{12}-n \parallel f_{12}-f$. This observation can be utilized to interpret the immersion edges of a circular net with constant negative Gau{\ss} curvature as diagonals in immersion quadrilaterals of K-nets (with equal lengths per K-net quadrilateral but in general different edge length for each circular net edge). Circular nets of constant negative Gau{\ss} curvature have been defined in \cite{Konopelchenko:1999te,Bobenko:2010eg} and naturally carry over to edge-constraint nets.

\begin{definition}
An edge-constraint net $(f,n): D \subset \Z^2 \to \R^3\times\Stwo$ that is a circular net with Gau{\ss} curvature $\K = -1$ is called a \emph{cK-net}.
\end{definition}
It turns out we can explicitly define cK-nets using their construction from K-nets as observed above: In Appendix \ref{appendix:cKnets} we give a Lax pair representation for cK-nets based on the Lax pair for K-nets and show that the  associated family for K-nets gives rise to an associated family for cK-nets in a natural way.
\begin{theorem}
  Every cK-net and its associated family are of constant negative Gau\ss{} curvature.
  \begin{proof}
    The way we define cK-nets here makes the statement for cK-nets tautological. The proof that their associated family has the same constant negative Gau{\ss} curvature can be found in Appendix \ref{appendix:cKnets}.
  \end{proof}
\end{theorem}
\begin{figure}[htb]
  \centering
  \includegraphics[width=\hsize]{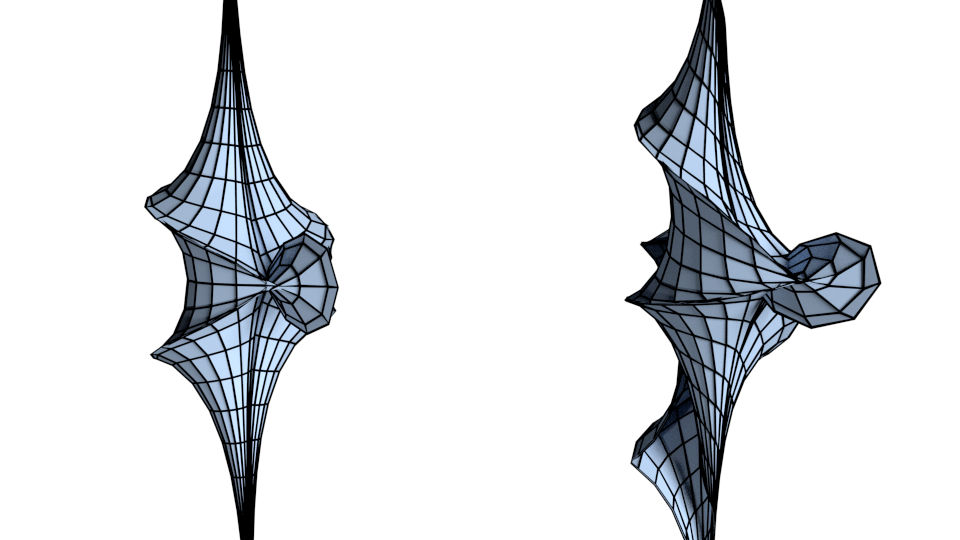}
  \caption{A cK-net Kuen surface and a member of its associated family.}
  \label{fig:kuen}
\end{figure}
Figure~\ref{fig:kuen} left shows a cK-net Kuen surface---a B\"acklund transform of the pseudosphere shown in Figure~\ref{fig:threePseudospheres} to the right. Such cK-net pseudospheres have been discussed by Konopelchenko and Schief \cite{Konopelchenko:1999te} as well as Bobenko, Pottmann, and Wallner \cite{Bobenko:2010eg}.
 Figure~\ref{fig:kuen} right shows a member of the associated family of the cK-net Kuen surface. Note, that the immersion quadrilaterals are no longer circular.

We can also create pseudospheres of revolution that have one asymptotic and one curvature line. The middle net of Figure \ref{fig:threePseudospheres} is generated by first constructing an asymptotic line with two degrees of freedom that are then used to impose rotational symmetry and constant negative Gau{\ss} curvature, respectively.

\subsection{Discrete developable surfaces}
\label{sec:developable}
We define developability exactly as in the smooth setting.

\begin{definition}[Developable edge-constraint net]
An edge-constraint net $(f,n)$ is called \emph{developable} if every quad has vanishing Gau{\ss} curvature ($\K = 0$), i.e., the Gau{\ss} map has zero area.
\end{definition}

As in the smooth setup, if the Gau{\ss} map is constant then the shape operator vanishes, otherwise there exists exactly one nonzero principal curvature $2 \mathcal{H}$ and corresponding curvature line, which is in fact even parallel to $n_x$ and $n_y$.

\begin{lemma}
Consider a single quad from a developable edge-constraint net $(f,n)$ with admissible projection direction $N \in \Stwo$ such that $N \perp \spann\{n_x, n_y\}$. If they exist, the nonzero principal curvature and curvature line are invariant to the choice of $N$.

\begin{proof} If $n$ is nonconstant then $\spann\{n_x, n_y\}$ is one dimensional. Consider the reduced coordinates where $n_x, n_y \in \R^3$ are both multiples of $e_1$, the first standard basis vector of $\R^3$, and let $f_x, f_y \in \R^3$. The set of admissible projection directions is then parametrized by an $\Sone$ degree of freedom in $e_2, e_3$. In these coordinates, direct computation (using that $f_x \cdot n_y = n_x \cdot f_y$) completes the proof.
\end{proof}
\end{lemma}

\begin{figure}
  \centering
  \includegraphics[width=0.55\hsize]{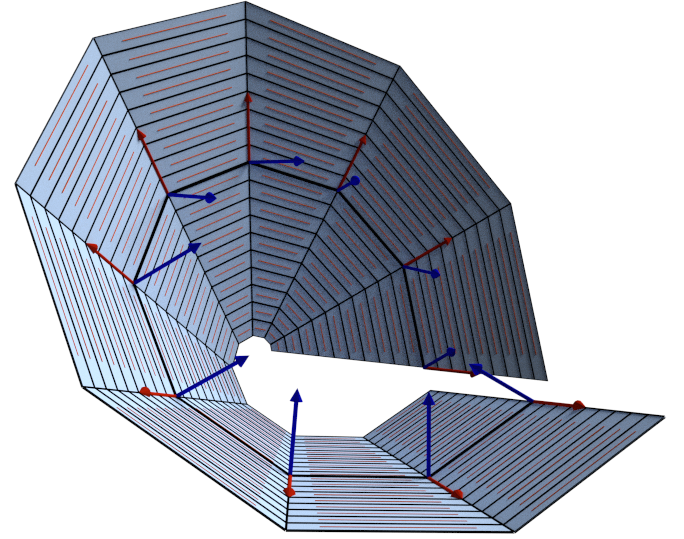}
    \includegraphics[width=0.4\hsize]{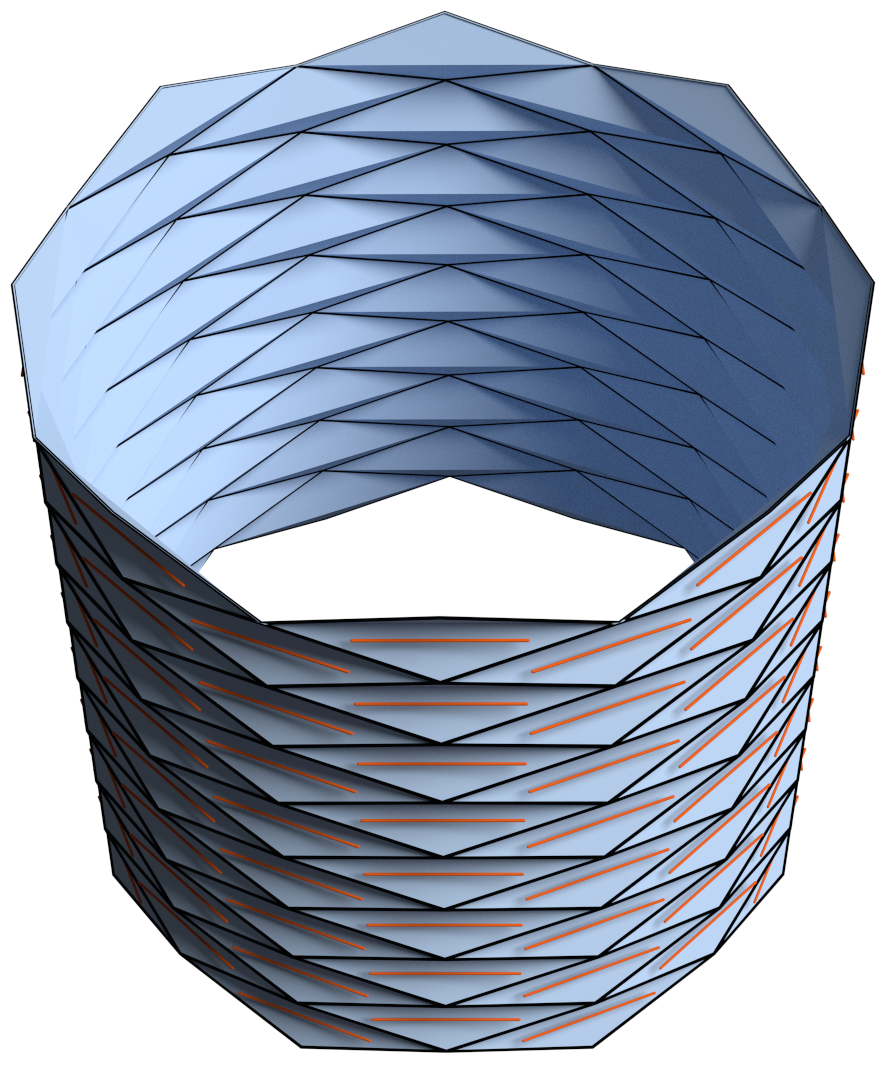}
  \caption{Left: A circular developable edge-constraint net generated from a discrete parallel framed polygonal helix (thick black). The Gau{\ss} map (blue) is constant along the discrete rulings (red). Observe that the net is built from planar strips quadrangulated by isosceles trapezoids and that the curvature line field (orange) is constant within each strip and parallel to the helix edge.
  Right: The Schwarz Lantern as a developable edge-constraint net with its curvature line field (orange).}
  \label{fig:developableNets}
\end{figure}

Surfaces of planar strips have been considered as discrete developable as they can obviously be unfolded into the plane \cite{Liu:2006tm}; such immersions correspond to developable curvature line edge-constraint nets, which are characterized by a discrete analogue of parallel framed curves \cite{Bishop:1975wg} (for example see Figure \ref{fig:developableNets} left):

A polygonal curve $\alpha$ with vertices $\alpha_0, ..., \alpha_k$ and two orthonormal vectors $(u_0, n_0)$ anchored at $\alpha_0$ extend to a unique \emph{discrete parallel frame} along $\alpha$; simply reflect through the perpendicular bisector planes of each edge of $\alpha$. \footnote{This discrete parallel frame can also be understood as being generated from rotations about the curve binormal since the composition of two reflections is a rotation.} Then $u_{i+1} - u_i$ and $n_{i+1} - n_i$ are both parallel to $\alpha_{i+1} - \alpha_i$ for all $i = 0, ..., k-1$. To extend the polygonal curve $\alpha$ to a developable edge-constraint net $(f,n)$, fix a sampling $y_j$ of the real line and define $f: (i,j) \mapsto \alpha_i + y_j \, u_i$ with Gau{\ss} map $n_i$. The resulting net $(f,n)$ is then in fact circular.

Conversely, for a surface $M$ built from a collection of planar strips with intersection lines $u$ one can find a discrete parallel framed curve $(\alpha, u, n)$ giving rise to a developable circular edge-constraint net whose immersion realizes $M$: choose an initial point $\alpha_0$ on an initial line $u_0$ and a unit normal $n_0 \perp u_0$. The reflection property then gives rise to unique $\alpha$ and $n$. It turns out that the curvature line is invariant to the initial choice of $n_0$ (it is always parallel to $\alpha$), while the mean curvature is not (which can be interpreted as extra information on how the developable surface locally bends).

Examples of developable edge-constraint nets that are not in curvature line parametrization arise from the associated family of the constant mean curvature discrete isothermic cylinder; this family contains the well known Schwarz Lantern \cite{Morvan:2008fe} (see Figure \ref{fig:developableNets} right) as an immersion with vertex normals that coincide with those of the smooth cylinder.

\section{Towards a conformal perspective}
\label{sec:conformal}
In the smooth setting two conformal immersions for a manifold $M$, $f, \tilde f:M \to \R^3 \cong \tup{Im}\Quat$, are said to be spin-equivalent if there exists a \emph{spin transformation} $\lambda: M \to \Quat^*$, such that $\tup{d}\tilde f = \bar \lambda \tup{d}f\lambda$; the surface normal $n$ transforms as $\tilde n = \lambda^{-1} n \lambda$. Geometrically, spin transformations correspond to stretch rotations of the tangent plane at every point. Therefore, they are conformal mappings and for simply connected domains any two surfaces which are conformally equivalent are related via a spin transformation. Kamberov, Pedit, and Pinkall \cite{kamberov1998} showed using spin transformations that one can classify all Bonnet pairs on a simply connected domain. Bonnet pairs are immersed surfaces that have the same metric and mean curvature but are not rigid body motions of each other.

We can define a discrete spin transformation by "stretch-rotating" the normal transport quaternions (Definition \ref{def:normalTransport}) of an edge-constraint net.

\begin{definition}[Discrete Spin Transformation]
\label{def:spinTransfo}
Let $(f,n)$ be an edge-constraint net with quad graph $\G$. The \emph{spin transformation} is a map $\lambda: \G \to \Quat^*$ which transforms $(f,n)$ to $(\tilde f, \tilde n)$. The normal at each vertex and the normal transport quaternions transform by:
\begin{equation}
\tilde n = \lambda^{-1} n \lambda, \quad \tilde \phi = \bar \lambda \phi \lambda_1, \quad \mathrm{and} \quad \tilde \psi = \bar \lambda \psi \lambda_2.
\end{equation}
If the immersion of the spin transformed quadrilateral closes (i.e., $\tilde \phi + \tilde \phi_2 + \tilde \psi + \tilde \psi_1$ is real) then one can construct a new edge-constraint net via 
\begin{equation}
\tilde f_1 - \tilde f = \Im\tilde \phi \quad \mathrm{and} \quad \tilde f_2 - \tilde f = \Im\tilde \psi.
\end{equation}
\begin{proof}
The Gau\ss\, map $\tilde n$ is still unit length since conjugation by a quaternion corresponds to a global rotation, so lengths are preserved. The edge constraint is satisfied since for an arbitrary edge (here denoted as first lattice shifts) we have
\begin{equation}
\tilde \phi^{-1} \tilde n \tilde \phi = \lambda_1^{-1} \phi ^{-1} \bar \lambda^{-1} \lambda^{-1} n \lambda \bar\lambda \phi \lambda_1 = \lambda_1^{-1} \phi^{-1} n \phi  \lambda_1 = \lambda_1^{-1} (-n_1) \lambda_1 = -\tilde n_1.
\end{equation}
\end{proof}
\end{definition}

The spin transformation is invertible in the following sense, yielding a notion of discrete conformity.
\begin{lemma}
If $(\tilde f, \tilde n)$ is a spin transform of $(f,n)$ with $\lambda$ then $(f,n)$ is a spin transformation of $(\tilde f, \tilde n)$ with $\lambda^{-1}$.
\end{lemma}

\begin{figure}
  \centering
  \includegraphics[width=.7\hsize]{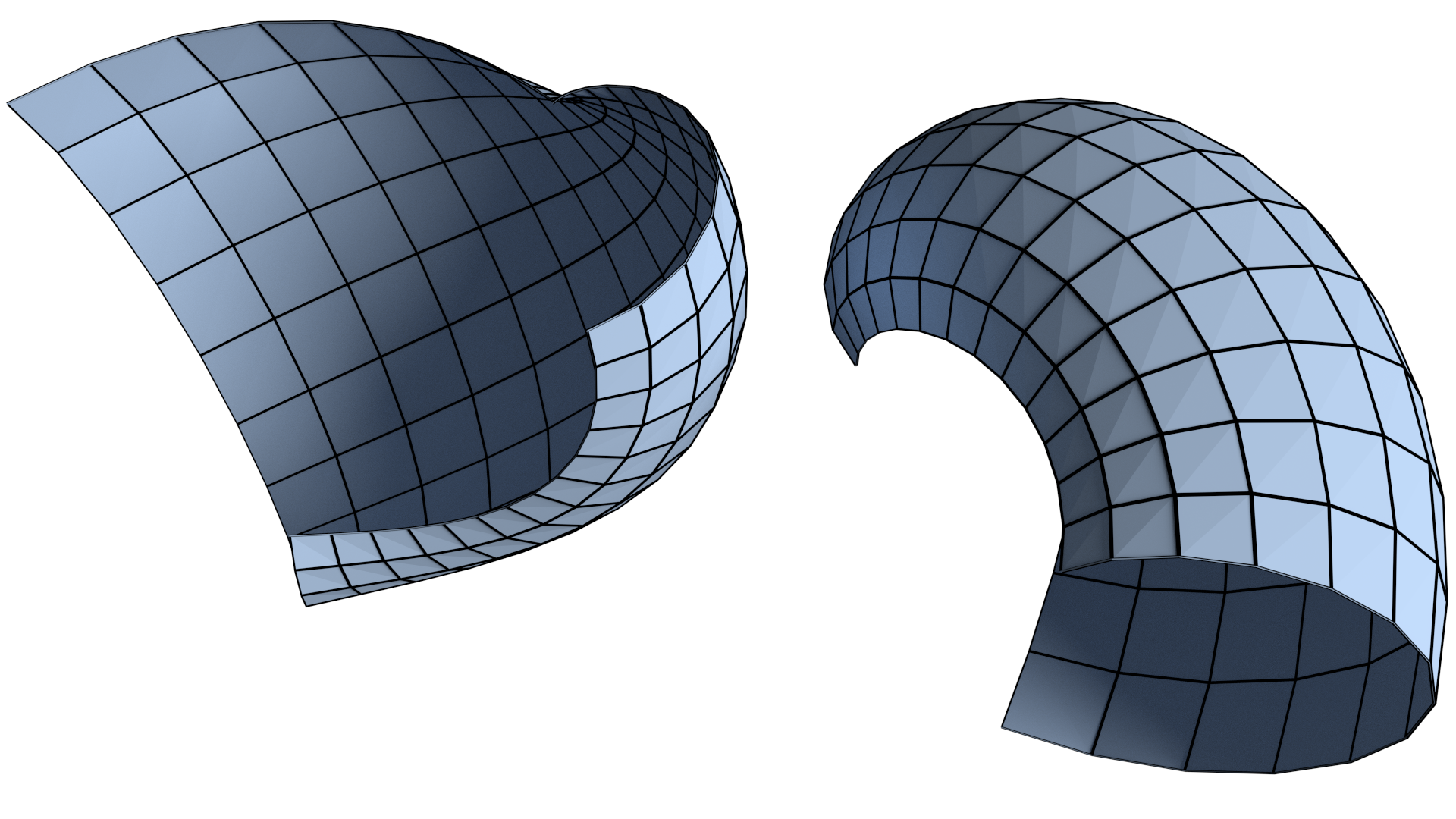}
  \caption{A Bonnet pair: two edge-constraint nets with the same mean curvature per quad that are discretely conformally equivalent but not rigid body motions of each other.}
  \label{fig:bonnetPairs}
\end{figure}

\begin{definition}[Discrete Conformal]
Two edge-constraint net quads are \emph{discretely conformally equivalent} if they are spin-transformations of each other.
\end{definition}

This spin transformation can give rise to edge-constraint net Bonnet pairs, an example is shown in Figure \ref{fig:bonnetPairs}. The Darboux transformations of discrete isothermic nets \cite{Hoffmann:1999vb} are also spin transformations, the proof is along the same lines as that of the following theorem.

\begin{theorem}
The nets in the associated family of a discrete isothermic minimal net are conformal to each other.
\begin{proof}
Perform the spin transformation with $\lambda = \cos(\frac{\alpha}{2}) - \sin(\frac{\alpha}{2})n$ on the discrete isothermic minimal nets given by the Weierstrass representation (Definition \ref{def:discreteWeierstrass}) and see that one recovers the associated family edge (Definition \ref{def:minimalAssocFamily}).
\end{proof}
\end{theorem}

\section{Lax pair edge-constraint nets}
\label{sec:LaxPairedgeConstraint}
The language of moving frames combined with the theory of integrable systems offers a powerful tool with which to unify and systematically discretize parametrized surface theory for particular types of special parametrizations \cite{Bobenko:1999us,Bobenko:2008tn}. We used this explicitly for showing that discrete isothermic constant mean curvature nets \cite{Pedit:1995wp,Bobenko:1999us,Hoffmann:1999vm} and their associated families are constant mean curvature edge-constraint nets. Although not explicitly used here, K-nets and their associated families also possess algebraic formulations in this integrable framework \cite{Bobenko:1996vq,Bobenko:1996ug}. 
In general, this description is based on the existence of a Lax pair governing the structure equation of the surface---which can then be integrated via the Sym--Bobenko formula \cite{Bobenko:1994tv,Sym:1985kl}---or in the discrete case alternatively by the 3D consistency of the governing equation \cite{Bobenko:2008tn}. We now state a condition on Lax matrices (expressed as elements in the space of invertible quaternions $\Quat^*$) that guarantees the resulting nets will be edge-constraint. The condition is not very strong, and after an appropriate gauge transformation many discrete integrable systems related to surface theory exhibit such representations.

\begin{theorem}
Let $\lambda(\alpha)$ be a spectral parameter depending on a real parameter $\alpha$. Let $\Phi \in \Quat^*$ be a moving frame defined by shifts $\Phi_1 := U(\lambda) \Phi$ and $\Phi_2 := V(\lambda) \Phi$ (starting from a fixed $\Phi$), where $U(\lambda),V(\lambda) \in \Quat^*$ are the $\emph{Lax matrices}$ satisfying the compatibility condition $V_1(\lambda) U(\lambda) = U_2(\lambda) V(\lambda)$. Let $s \in \R$ be an arbitrary coefficient. Then the family of discrete contact element nets $(f^\alpha,n^\alpha)$ given by the \emph{Sym--Bobenko} formula
\begin{eqnarray}
	n^\alpha &:=& \Phi^{-1} \qk \Phi \nonumber \\
	\qf^\alpha &:=& s\, \Phi^{-1}\frac{d}{d\alpha}\Phi\vert_{\lambda(\alpha)} + t n = s\Phi^{-1}\Phi_\alpha + tn, \label{eq:genSymBobenko} \\
	f^\alpha &:=& \Im \qf, \nonumber
\end{eqnarray}
for some $t \in \R$ are edge-constraint nets if and only if the Lax matrices $U(\lambda)$ and $V(\lambda)$ depend on $\lambda$ only in their off-diagonal entries. The edge-constraint is encoded in the relationships
\begin{equation}
\label{eq:laxTransports}
n_1 = -(\qf_1 - \qf) n (\qf_1 - \qf)^{-1} ~ \mathrm{and} ~ n_2 = -(\qf_2 - \qf) n (\qf_2 - \qf)^{-1}.
\end{equation}
\begin{proof}
The proof is equivalent for both lattice directions, so we provide details for the first one, resulting in a condition on the $U(\lambda)$ Lax matrix:
Up to a global rotation by $\Phi$ we find
\begin{eqnarray*}
(\qf_1 - \qf) n &=& (s\, U^{-1}U_\alpha + t(U^{-1} \qk U - \qk))\qk \\
&=& U^{-1} ( s\, U_\alpha + t(\qk U - U \qk)\qk \\
&=& U^{-1} \left( s\, U_\alpha \qk + t (\qk U \qk + U) \right) \\
&=& U^{-1} \left(s\, U_\alpha \qk + t ((-\qk U) (-\qk) + (-\qk U) (-\qk U)^{-1}U) \right) \\
&\stackrel{!}{=}& U^{-1} (- \qk U) (\qf_1 - \qf)\\
&=& -n_1 (\qf_1 - \qf)
\end{eqnarray*}
precisely when:
\begin{equation}
U_\alpha \qk = -\qk U_\alpha,
\end{equation}
which is equivalent to $U$ having only off-diagonal entries dependent on $\lambda(\alpha)$.
\end{proof}
\end{theorem}

\section{Acknowledgements}
We thank Julia Plehnert and Henrik Schumacher for helpful discussions, particularly with helping to resolve the definitions with a developable theory of edge-constraint nets.

\newpage
\appendix
\newcommand{\quadmatrix}[4]{\left(\begin{array}{cc}#1&#2\\#3&#4\end{array}\right)}
\section{cK-net primer}
\label{appendix:cKnets}

Following \cite{Bobenko:1996ug} a Lax representation for K-nets is given by the $SU(2)$ matrices
\begin{equation}
  \label{eq:KLaxPair}
  \begin{array}{rcl}
    U &=& \quadmatrix{\cot(\frac{\sigma_u}2)\frac {H_1}{H}}{i \lambda}{i \lambda}{\cot(\frac{\sigma_u}2)\frac {H}{H_1}}\\
    V &=& \quadmatrix{1}{\frac{i}{\lambda}\tan(\frac{\sigma_v}2)H_2H}{\frac{i}{\lambda}\tan(\frac{\sigma_v}2)\frac1{H_2H}}{1}    
  \end{array}
\end{equation}
depending on a (real) spectral parameter $\lambda$ with $H = e^{i h}$ for $h \in \R$ and the matrix problem
\begin{equation}
  \label{eq:ZeroCurvature}
\Phi_1 = U \Phi, \quad \Phi_2 = V \Phi.
\end{equation}
The integrability condition $V_1U = U_2V$ is then equivalent to $h$ solving the Hirota equation \cite{Hirota:1977tx}
\begin{equation}
e^{i(h_{12}+h)} - e^{i(h_1+h_2)} =\tan\frac{\sigma_u}2\tan\frac{\sigma_v}2 \left( 1-e^{i(h+h_1+h_{12}+h_2)} \right).
\end{equation}

Given the observation that the diagonals in K-nets satisfy the edge-constraint one can consider matrices
\begin{equation}
\mathcal L = V_1U = \quadmatrix{\cot\frac{\delta}2\frac{H_1}H + \tan\frac{\delta}2 H_1H_{12}}{i(\lambda -\frac{HH_{12}}\lambda)}{i(\lambda - \frac1{\lambda HH_{12}})}{\cot\frac{\delta}2\frac{H}{H_1} + \tan\frac{\delta}2 \frac{1}{H_1H_{12}}}
\end{equation}
by choosing $\delta = \sigma_u = -\sigma_v$. Setting aside the fact that $\mathcal L$ arises as the product of two matrices along edges of a K-net, we can investigate the compatability condition for  $\mathcal L$ assigned to edges of a lattice. After relabeling the entries we have
\begin{equation}
\begin{array}{rcl}
  \label{eq:LaxPair}
L &=&  \quadmatrix{\cot\frac{\delta_1}2\frac{l}s + \tan\frac{\delta_1}2 ls_1}{i(\lambda -\frac{ss_{1}}\lambda)}{i(\lambda - \frac1{\lambda ss_{1}})}{\cot\frac{\delta_1}2\frac{s}{l} + \tan\frac{\delta_1}2 \frac{1}{ls_{1}}} \\
M &=&  \quadmatrix{\cot\frac{\delta_2}2\frac{m}s + \tan\frac{\delta_2}2 ms_2}{i(\lambda -\frac{ss_{2}}\lambda)}{i(\lambda - \frac1{\lambda ss_{2}})}{\cot\frac{\delta_2}2\frac{s}{m} + \tan\frac{\delta_2}2 \frac{1}{ms_{2}}}
\end{array}
\end{equation}
with unitary variables $s$ at vertices of a square lattice and $l$ and $m$ on
edges in first and second lattice directions.
The compatibility condition
\begin{equation}
  \label{eq:integrabilityCondition}
 M_1L =L_2M
\end{equation}
implies $\det M_1\det L = \det L_2 \det M$ and in the spirit of the
K-net case we thus assume $\delta_1$ is constant in the second
lattice direction and $\delta_2$ is constant in the first one.
Then Equation (\ref{eq:integrabilityCondition}) can be solved, i.e., given $s, s_1, s_2, l, m, \delta_1$, and $\delta_2$ then $m_1, l_2,$ and $s_{12}$ are uniquely determined. 
However, in order to be able to have arbitrary edge lengths for the resulting net, one needs to allow for $\sin\delta_i>1$ and thus complex $\delta_i$. To keep $L$ and $M$ quaternionic $l$ and $m$ will no longer be unitary but have absolute value
\begin{equation}
\sqrt{\frac{\cos\left(\rho_i-\arg\tan\frac{\delta_i}2\right)+\cos\left(\rho-\arg\tan\frac{\delta_i}2\right)}{\cos\left(\rho_i+\arg\tan\frac{\delta_i}2\right)+\cos\left(\rho+\arg\tan\frac{\delta_i}2\right)}}
\end{equation}
where $s=e^{i \rho}$.

\begin{theorem}
  Let $\Phi:\Z^2\to\R^3$ be a solution to $\Phi_1 = L\Phi$, $\Phi_2 =
  M\Phi$ with $L$ and $M$ as in Equation (\ref{eq:LaxPair}) solving the
  integrability condition Equation (\ref{eq:integrabilityCondition}). Then
  $f:\Z^2\to \R^3$ and $n:\Z^2\to\Stwo$ given by
\begin{equation}
  \label{eq:Sym}
f = 2\Phi^{-1}\frac{\partial}{\partial t}\Phi ~ \mathrm{and} ~  n = \Phi^{-1}\qk\Phi
\end{equation}
is an edge-constraint net of Gau{\ss} curvature $K=-1$. If $\lambda =
1$, $f$ is circular. Moreover any circular edge-constraint
quadrilateral with parallel normals and $K=-1$ can be generated this
way.
\end{theorem}
\begin{proof}
  Circularity in the case of $\lambda = 1$ and the Gau{\ss} curvature
  $K=-1$ for any real $\lambda$ can be computed directly.  To see that
  every quadrilateral arises this way, one can show that the quad
  with its normals is uniquely defined by three vertices and a
  normal. However this is exactly the initial data one prescribes for
  the compatibility condition of the Lax pair. Since both problems
  have a unique solution they must coincide.
\end{proof}

\newpage
\bibliography{mybib}

\todos

\end{document}